\documentclass[11pt,a4paper,reqno]{amsart}
\usepackage[margin=1.1in]{geometry}
\usepackage{amssymb,amsmath,graphicx,amsfonts}
\usepackage{mathrsfs}
\usepackage{multirow}
\usepackage{graphicx,epsfig}
\usepackage{color}
\usepackage{enumerate,enumitem}
\usepackage{empheq}
\usepackage{cases}
\usepackage{cite}
\usepackage{mathrsfs}
\usepackage{amsgen}
\usepackage{amscd}
\allowdisplaybreaks

\usepackage{booktabs}
\usepackage{amsmath}
\usepackage{latexsym}
\usepackage{amssymb,amsmath,graphicx,amsfonts,euscript}
\usepackage{amsfonts,bm,hyperref,subcaption}
\usepackage{enumerate}
\newtheorem{theorem}{Theorem}[section]

\newtheorem{lemma}[theorem]{Lemma}
\newtheorem{proposition}[theorem]{Proposition}
\newtheorem{remark}[theorem]{Remark}
\theoremstyle{definition}

\newtheorem{example}[theorem]{Example}

\numberwithin{theorem}{section}
\numberwithin{equation}{section}

\newcommand{\Tr}{\text{Tr}}

\newcommand{\Man}{\mathcal{M}}

\definecolor{link}{rgb}{0.45,0.51,0.67}

\renewcommand{\div}{\text{div}}
\newcommand{\A}{{\bf A}}

\newcommand{\M}{{\bf M}}

\newcommand{\g}{{\bf g}}

\newcommand{\ka}{\boldsymbol{\kappa}}
\newcommand{\kas}{\boldsymbol{\kappa^*}}
\newcommand{\bv}{{\bf v}}
\newcommand{\w}{{\bf w}}

\newcommand{\dkappa}{{\bf d_\kappa}}

\newcommand{\ekappa}{{\bf e_\kappa}}
\newcommand{\gs}{{\bf g^*}}
\newcommand{\Eh}{\mathcal{E}_h}
\newcommand{\N}{\mathcal{N}}
\DeclareMathOperator{\sff}{\mathrm{I\!I}}

\begin{document}

\title[]{Convergence of Finite Element methods for Ricci Flow}\thanks{This work was supported by the Research Grants Council of Hong Kong (Project No. PolyU/GRF15303022 and PolyU/RFS2324-5S03), NSFC Key Program (project no. 12231003), and the U.S. National Science Foundation award DMS-2533499.}

\author[]{Guangwei Gao,\,\, Evan S. Gawlik,\,\, and\,\, Buyang Li}
\address{Guangwei Gao, Buyang Li: Department of Applied Mathematics, The Hong Kong Polytechnic University, Hong Kong.
{\rm Email address: {\tt guang-wei.gao@polyu.edu.hk} and {\tt buyang.li@polyu.edu.hk}
}}
\address{Evan S. Gawlik: Department of Mathematics and Computer Science, Santa Clara University, Santa Clara, USA.
{\rm Email address: {\tt egawlik@scu.edu}
}}



\maketitle

\begin{abstract}
    The convergence of a finite element discretization for the two-dimensional Ricci flow is proved. 
    In this method, the Ricci flow on a two-dimensional surface is formulated into solution-driven metric evolution, with the metric evolution driven by the Gauss curvature.  The Gauss curvature satisfies a parabolic equation that in turn depends on the metric, thereby enhancing the parabolic structure of the problem. The solution-driven metric evolution formulation is discretized by the finite element method, and the convergence of finite element approximations is proved by adapting the matrix-vector formulation developed in the literature initially for studying solution-driven surface evolution in extrinsic curvature flow.
    In addition to its convergence, the proposed method also preserves important geometric structures of the Ricci flow at the discrete level, such as area conservation and the Gauss-Bonnet theorem.
    Extensive numerical experiments are presented to demonstrate the convergence of the proposed method as well as the simulation of Ricci flow.
\end{abstract}

\setlength\abovedisplayskip{4pt}
\setlength\belowdisplayskip{4pt}


\section{Introduction}\label{sec:intro}
The Ricci flow, introduced by Richard Hamilton in \cite{hamilton1982three} in 1982, is a geometric evolution equation that deforms the Riemannian metric on a manifold according to its induced curvature. Since then, the Ricci flow has been well-developed and has become a powerful tool in geometry and topology. In particular, Grigori Perelman built upon Ricci flow techniques and proved the Poincar\'{e} conjecture in his groundbreaking works \cite{perelman2002entropy, perelman2003finite, perelman2003ricci}.
The Ricci flow has a wide range of applications, including image processing \cite{sonn2014ricci}, 3D face matching and registration \cite{zeng20083d}, surface parametrization and re-meshing \cite{jin2006conformal}, and conformal brain mapping in medical imaging \cite{wang2011brain}.

For the numerical study of Ricci flow, Chow and Luo introduced a discrete version of the Ricci flow in \cite{chow2003combinatorial} based on discrete differential geometry.  It was further developed in \cite{jin2008discrete, gu2008computational}. In this method, the Gaussian curvature is approximated by the angle defect, and at each vertex, a scalar conformal factor evolves according to the angle defect. The long-time convergence to a discrete constant curvature metric for this method has been proven in \cite{chow2003combinatorial}.
{However, the convergence of this method to the exact solution of the Ricci flow remains an open question. In other words, it is not guaranteed that this scheme will provide an accurate approximation of the solution of Ricci flow.}  In many applications, this is not a serious concern, as the goal is often to construct a discrete conformal mapping from a given triangulation to one with prescribed discrete curvature, rather than to accurately approximate the Ricci flow. 
For numerical studies of the Ricci flow in certain special cases, we further refer to \cite{garfinkle2005numerical, rubinstein2005visualizing, taft2010intrinsic}.

Compared to Ricci flow, the numerical methods for other extrinsic curvature flows, such as mean curvature flow and Willmore flow, have been more extensively studied. One method that has been well-developed is the parametric finite element method (FEM), first introduced by Dziuk in \cite{dziuk1990algorithm} for approximating surface evolution under curvature flows and, subsequently, has been extensively used to approximate various geometric evolution equations \cite{dziuk2013finite, bonito2010parametric, bansch2005finite}. Later, the idea of artificial tangential velocity for evolving surfaces was introduced by Barrett, Garcke, and N\"urnberg in their seminal works \cite{barrett2007parametric, barrett2008parametric}. The artificial tangential velocity drives the mesh points to move tangentially on the surface, which does not alter the shape of the evolving surface but improves the quality of the surface mesh during evolution; see also \cite{duan2024new, hu2022evolving, bao2021structure, bao2024structure, bai2024convergence, elliott2017approximations}.

 
The convergence of Dziuk's scheme for curve shortening flow has been established in \cite{deckelnick1995approximation, li2020convergence, ye2021convergence}. 
A parametric FEM based on DeTurck flow techniques was developed by Elliott and Fritz in \cite{elliott2017approximations}, and its convergence for curve shortening flow was proved. 
A convergent numerical scheme for elastic flow was proposed in \cite{bartels2013simple}. 
An error estimate for planar curve evolution with triple junctions was established in \cite{pozzi2021motion}.
Moreover, convergence for curve evolution coupled to a reaction-diffusion equation was established in \cite{pozzi2017curve}, and error analysis for curve shortening flow coupled with lateral diffusion was presented in \cite{barrett2017numerical}.
Furthermore, the convergence of parametric FEMs for the mean curvature flow and the Willmore flow of closed surfaces in three-dimensional space was first established by Kov\'acs, Li, and Lubich in \cite{kovacs2019convergent, kovacs2021convergent}, based on a novel reformulation of these geometric flows involving the evolution equations for the normal vector and the mean curvature. 
Stability and error analysis for advection-diffusion equations on moving surfaces was conducted in \cite{deckelnick2018stability}, and convergence analysis for the interaction between mean curvature flow and diffusion on closed surfaces was established in \cite{elliott2022numerical}.

The Ricci flow, as an intrinsic curvature flow, presents significant challenges for numerical discretization compared to the extrinsic geometric flows studied in the literature. 
One major difficulty lies in the discretization of the Riemannian metric, and computing the Gaussian curvature directly from the metric.  This involves a nonlinear second-order differential operator and is particularly challenging in a spatially discrete setting { since, when applied to piecewise smooth metrics, this differential operator must be interpreted in the sense of distributions~\cite{berchenko2024finite}}.
A finite element discretization of the Ricci-DeTurck flow has been studied in \cite{fritz2015numerical}. In this method, the Riemannian metric \( g \) on an \( n \)-dimensional embedded manifold \( \Gamma \) is extended to a Riemannian metric \( G \) on a tubular neighborhood \( \Gamma_{\delta} \subset \mathbb{R}^{n+1} \). Consequently, the Ricci-DeTurck flow can be reformulated in the Euclidean coordinates of the ambient space \( \mathbb{R}^{n+1} \), and then discretized using the surface finite element method. 
{Although various numerical simulations are reported in \cite{fritz2015numerical}, stability and error analyses for this method remain open problems. } 
{One obstacle to the numerical analysis of Ricci flow is the lack of a weak formulation that is both amenable to discretization and capable of establishing rigorous stability and error estimates.
This differs from extrinsic curvature flows; for instance, a suitable weak formulation for mean curvature flow was already proposed by Dziuk in 1990 \cite{dziuk1990algorithm}, and the numerical analysis of this formulation has since been investigated extensively \cite{deckelnick1995approximation, li2020convergence, ye2021convergence}.
Another challenge arises from the intrinsic nature of Ricci flow: both integrals and differential operators are defined with respect to the evolving metric, which is itself the unknown. 
As a result, one must approximate metric-dependent integration and differentiation in an intrinsic way and, moreover, charactrize the errors induced by the approximation of the evolving metric. This again differs from the extrinsic setting, where the relevant integrals and differential operators are computed with respect to an evolving surface in $\mathbb{R}^3$, and the corresponding error analysis is by now well developed \cite{kovacs2017convergence,kovacs2019convergent,kovacs2021convergent}.
} 

Our construction is based on Regge finite elements, a recently developed family of finite elements for discretizing symmetric $(0,2)$-tensor fields on simplicial triangulations \cite{regge1961general, christiansen2011linearization, li2018regge, christiansen2023extended, christiansen2023finite}.
In the lowest-order case, the Regge finite element discretizes the metric tensor through edge lengths, allowing curvature to be approximated by angle deficits. This was further improved by Gawlik in \cite{gawlik2020high}, developing a high-order Gaussian curvature approximation based on high-order Regge elements. More recently, in \cite{berchenko2024finite}, this approach has been used to approximate the connection 1-form (Levi-Civita connection), see also \cite{gopalakrishnan2023analysis}.
Notably, in \cite{gawlik2019finite}, Gawlik applied this approach to the Ricci flow and proposed a finite element discretization of the Ricci flow of arbitrary polynomial order. Unlike \cite{chow2003combinatorial, fritz2015numerical}, \cite{gawlik2019finite} reformulates the Ricci flow as a coupled system by treating the curvature $\kappa$ and the metric $g$ as independent variables (see Section \ref{sec:Ric} for details). A notable advantage of this approach is that it eliminates the need to compute the curvature from the metric, which involves a nonlinear second-order differential operator. 
Additionally, the scheme proposed in \cite{gawlik2019finite} preserves several important geometric structures of the Ricci flow at the discrete level, including the Gauss-Bonnet theorem and area conservation (see Theorem \ref{thm:GeoStruc}).

{The main contribution of this paper is to provide a rigorous numerical analysis for the FEM for Ricci flow proposed in \cite{gawlik2019finite}. Specifically, we prove stability and error estimates for the discretization \cite[(9)-(11), Section 2]{gawlik2019finite}, based on the following two key observations. The first key observation is that} the reformulated system can be viewed as a solution-driven metric evolution, where the metric evolves according to the Gaussian curvature, which satisfies a parabolic equation that, in turn, depends on the metric. This allows us to analyze the error by adapting the matrix-vector formulation developed in \cite{kovacs2017convergence, kovacs2019convergent}, which was initially introduced for studying solution-driven surface evolution in extrinsic curvature flow.  
{The second key observation is that}  the evolution equation for the Gaussian curvature in the coupled system (see \eqref{eq:Ref-Ricci-b}) exhibits a parabolic structure. This enables us to rigorously prove the convergence of the finite element discretization for arbitrary polynomial orders.
Specifically, if we use the Regge element of degree \( r \ge 0 \) to discretize the Riemannian metric, and the Lagrange element of degree \( q \ge 1 \) to discretize the Gaussian curvature, we have the following error estimate: 
\[
\|g_h(t) -  g^{(-l)}(t) \|_{L^p(M_h)} + \|\kappa_h(t) -  \kappa^{(-l)}(t) \|_{L^2(M_h)} \lesssim (\ln(\frac{1}{h}))^{\bar q} h^{q+1} + \ln(\frac{1}{h}) h^{r+1}, \quad t \in [0,T], 
\]
where $g_h -  g^{(-l)}$ and $\kappa_h -  \kappa^{(-l)}$ denote the errors of metric and curvature respectively, with $g^{(-l)}$ and $\kappa^{(-l)}$ being the inverse lifts of $g$ and $\kappa$ onto the triangulated manifold $M_h$.
Here, $h$ is the mesh size, $p > 2$,  and  $\bar{q} = 1$ when $q = 1$, and $\bar{q} = 0$ otherwise. 
In addition, we can also establish the following optimal $L^2$ error estimate when $r \geq 1$ and $q \geq 1$:
\[
\| g_h(t) -  g^{(-l)}(t) \|_{L^2(M_h)} + \| \kappa_h(t) -  \kappa^{(-l)}(t) \|_{L^2(M_h)} \lesssim h^{q+1} + h^{r+1}, \quad t \in [0,T].
\]
See also the discussions in Remark \ref{rmk:L2-bound}, \ref{rmk:L2-stab-e_g}, \ref{rmk:L2-stab-e_k}. 
These error estimates further indicate that using this coupled system approach can simultaneously provide high-order approximations for both the Gaussian curvature and the Riemannian metric.

{ The intrinsic nature of the Ricci flow means that it evolves a Riemannian metric without reference to any embedding. To visualize the solution, we compute an isometric embedding of the evolving metric into Euclidean space. In Section~\ref{sec:embedding}, we propose a new continuous PDE for the embedding velocity and discretize it using a parametric FEM. This approach enables efficient simulation and visualization of Ricci flow solutions and can be viewed as an optional post-processing step.
}

The rest of this article is organized as follows. In Section \ref{sec:Ric}, we review the evolution equations for the Gaussian curvature of a metric evolving under the Ricci flow and formulate the system of equations to be discretized subsequently. In Section \ref{sec:FEM}, we propose the semi-discretization in space using Regge finite elements and Lagrange finite elements, and present the structure-preserving properties and the main theorem on the convergence of the numerical discretization.
The stability of the numerical discretization is proved in Section \ref{sec:stab}, and error estimates are provided in Section \ref{sec:err}. 
{A numerical embedding scheme for visualizing the numerical solutions is discussed in Section \ref{sec:embedding}.} 
Extensive numerical experiments are presented in Section \ref{sec:numerical} to demonstrate the convergence of the proposed method as well as the simulation of the Ricci flow. {Additional details of the results used in the stability analysis are provided in Appendix \ref{sec:Ritz} and \ref{sec:P-Stab}. }

\section{Evolution equation for Ricci flow}\label{sec:Ric}

\subsection{Basic notations}
Let $\Man$ be a smooth, oriented, two-dimensional manifold without boundary, embedded in $\mathbb{R}^3$. 
The tangent and cotangent bundles on $\Man$ are denoted by $T\Man$ and $T^*\Man$, respectively. 
Let $g$ be a smooth Riemannian metric on $\Man$. We use $\omega_g$ to denote the volume form on $\Man$ induced by $g$, and $\nabla_g$ to denote the Levi-Civita connection associated with $g$.
For a $(p,q)-$tensor $\sigma$, we use $\Tr_g$ to denote the contraction along the first two indices, using $g$ to raise or lower indices as needed. We denote $\div_g \sigma = \Tr_g \nabla_g \sigma $ and $\Delta_g \sigma = \div_g\nabla_g \sigma$. 
The pointwise inner product of two tensors $\sigma$ and $\rho$ with respect to the metric $g$ is denoted by $\langle \sigma, \rho \rangle_{g}$. For tensor fields $\sigma,\rho$ and scalar functions $u,v$ on $\Man$, we denote 
\[
\hspace{-2.2pt}
(u,v)_{g,\Man} = \int_{\Man} uv \omega_{g}, \quad 
(\nabla_g u, \nabla_g v)_{g, \Man} = \int_{\Man} \langle \nabla_g u, \nabla_g v \rangle_{g} \omega_g, \quad  
(\sigma, \rho)_{g, \Man} = \int_{\Man} \langle \sigma, \rho \rangle_{g} \omega_g.
\]
Let $\delta_{\mathbb{R}^3}$ denote the Euclidean metric on $\mathbb{R}^3$, which induces a metric on $\Man$ through the pullback under the inclusion map $i_{\Man}: \Man \hookrightarrow \mathbb{R}^3$, denoted by $i_{\Man}^* \delta_{\mathbb{R}^3}$, and is referred to as the induced Euclidean metric on $\Man$. We denote $\nabla$ and $\mathrm{d}x$ the Levi-Civita connection and the volume form associated with the induced Euclidean metric $i_{\Man}^{*}\delta_{\mathbb{R}^3}$. 
For tensor fields $\sigma,\rho$ and scalar functions $u,v$ on $\Man$, we denote 
\[
\hspace{-4.1pt}
(u,v)_{\Man} = \int_{\Man} uv \mathrm{d}x, \quad 
(\nabla u, \nabla v)_{\Man} = \int_{\Man} \langle \nabla u, \nabla v \rangle_{i_{\Man}^{*}\delta_{\mathbb{R}^3}} \mathrm{d}x, \quad  
(\sigma, \rho)_{\Man} = \int_{\Man} \langle \sigma, \rho \rangle_{i_{\Man}^{*}\delta_{\mathbb{R}^3}} \, \mathrm{d}x.
\]
We use $W^{k,p}(g, \Man)$ to denote the Sobolev space on $\Man$ associated with the metric $g$, and denote $L^{p}(g, \Man) = W^{0,p}(g, \Man)$ and $H^k(g, \Man) = W^{k,2}(g, \Man)$. The Sobolev space with respect to the metric $i_{\Man}^{*}\delta_{\mathbb{R}^3}$ is denoted by $W^{k,p}(\Man)$, $L^p(\Man) = W^{0,p}(\Man)$ and $H^k(\Man) = W^{k,2}(\Man)$.
We also use $L^p(0,T; W^{k,p}(\Man))$ to denote the Sobolev space with $L^p$ regularity in time and $W^{k,p}$ regularity in space.


\subsection{The system of equations used for discretization}
On a two-dimensional manifold $\Man$, the Ricci flow evolves a time-dependent metric $g(t)$ according to
\begin{equation}\label{eq:Ricci-flow}
    \partial_t g(t) = 2(\bar \kappa - \kappa(t)) g(t),
\end{equation}
where $\kappa(t) = \kappa(g(t))$ is the Gaussian curvature of $g(t)$, and \(\bar \kappa\) is a constant that can be set to zero or chosen as the average of initial curvature, given by  
\[
\bar \kappa = \frac{\int_{\Man} \kappa(0) \omega_{g(0)}}{\int_{\Man} \omega_{g(0)}}.
\]  
In such a case equation \eqref{eq:Ricci-flow} is referred to as the normalized Ricci flow.

The Gaussian curvature \(\kappa(t) = \kappa(g(t))\) on the right-hand side of \eqref{eq:Ricci-flow} is known to satisfy the following evolution equation; see, for instance, \cite[Proposition 2.1]{berchenko2024finite}, \cite[Proposition 2]{gawlik2019finite}, and the discussion in \cite[Section 2]{gawlik2019finite}.
\begin{lemma}[Evolution equation of curvature]\label{lm:evo-Gauss-curv} 
For a time-dependent metric $g(t)$ evolving under the Ricci flow \eqref{eq:Ricci-flow}, its Gaussian curvature  $\kappa(t) = \kappa(g(t))$ satisfies  
\begin{equation}\label{eq:evo-Gauss-curv} 
\frac{\mathrm{d}}{\mathrm{d} t} (\kappa(t), v)_{g(t), \Man} + 
(\nabla_{g(t)} \kappa(t), \nabla_{g(t)} v)_{g(t), \Man} = 0, \quad \forall v \in H^1(\Man), \partial_t v = 0. 
\end{equation} 
\end{lemma} 

Based on \eqref{eq:evo-Gauss-curv} and \eqref{eq:Ricci-flow}, we have reformulated the Ricci flow as the following coupled system: Find a time-dependent metric $g$ and a scalar function $\kappa$ on $\Man \times [0,T]$ satisfying 
\begin{subequations}\label{eq:Ref-Ricci}  
\begin{align}
\label{eq:Ref-Ricci-a} 
\frac{\partial }{\partial t} g(t) + 2( \kappa(t) - \bar \kappa) g(t) &= 0, \\ 
\label{eq:Ref-Ricci-b} 
\frac{\mathrm{d}}{\mathrm{d} t} ( \kappa(t), v )_{g(t), \mathcal{M}} + 
( \nabla_{g(t)} \kappa(t) , \nabla_{g(t)} v )_{g(t), \mathcal{M}} &= 0, \quad \forall v \in  H^1(\Man), \partial_t v = 0.
\end{align}
\end{subequations}
In this coupled system, the Riemannian metric evolves according to its Gaussian curvature, while the evolution equation for the Gaussian curvature exhibits a parabolic structure and, in turn, depends on the metric.
This system is complemented with the initial data $g(0)$ and $\kappa(0)$. 
The numerical discretization will be based on \eqref{eq:Ref-Ricci}. 

Throughout this paper, we denote by $C$ and $h_{0}$ two generic positive constants which are different at different occurrences, possibly depending on the exact solution and $T$, but are independent of the mesh size $h$ and $t\in[0,T]$. The notation $X \lesssim Y$ means $X \leq C Y$ for some constant $C$, and $X \eqsim Y$ means $X \lesssim Y$ and $Y \lesssim X$.

\section{Finite element discretization}\label{sec:FEM}

Let $M_h$ be a piecewise flat, quasi-uniform triangular surface whose vertices lie on $\Man$ , with mesh size $h$ and approximating $\Man$.
We assume that $M_h \subset D_{\varepsilon}(\mathcal{M}) = \{ p \in \mathbb{R}^3 \mid {\rm dist}(p, \mathcal{M}) \leq \varepsilon \}$ for some sufficiently small $\varepsilon > 0$, ensuring the existence of a bijective mapping $a: M_h \rightarrow \Man$ satisfying  
\[
    p - a(p) = \pm \lvert p - a(p) \rvert n(a(p)), \quad \forall p \in M_h,
\]
where $n: \Man \rightarrow \mathbb{R}^3$ denotes the outer unit normal vector of $\Man$.

We denote by $\mathcal{T}_h$ the set of triangles in $M_h$ and by $\mathcal{E}_h$ the set of its edges.
Any function $f$ on $M_h$ can be lifted to a function $f^{l}$ on $\Man$, defined by $f^{l} = f \circ a^{-1}$. 
The inverse lift of a function $\phi$ on $\Man$ can be expressed as $\phi^{(-l)} = \phi \circ a$ on $M_h$.
Any $(0,2)$ tensor field $\sigma$ on $\Man$ can be pulled back to $M_h$, defined by $\sigma^{(-l)} = a^{*} \sigma$, where $a^{*}$ denotes the pullback under the mapping $a:  M_h \rightarrow \Man$.

On a triangle $K \subset M_h$, there exists a coordinate chart such that the pair $(\partial_1, \partial_2)$ forms a constant orthonormal frame of the tangent bundle $TK$ with respect to the induced Euclidean metric $i_{K}^* \delta_{\mathbb{R}^3}$. 
Let $S_2^0(K)$ denote the space of all symmetric $(0,2)$-tensor fields on $K$. For any $\sigma \in S_2^0(K)$, it can be represented as a $2 \times 2$ symmetric matrix-valued function on $K$, defined by $(\sigma_{ij})_{1 \leq i,j \leq 2} = (\sigma(\partial_i, \partial_j))_{1 \leq i,j \leq 2}$, and this matrix representation is also referred to as $\sigma$. We also use $\sigma^{-1}$ and $\det(\sigma)$ to denote its inverse matrix and determinant.
We say that $\sigma \in H^1S_2^0(K)$ if each component of its matrix representation belongs to $H^1(K)$. Similarly, we say that $\sigma \in P_rS_2^0(K)$ if each component of its matrix representation is a polynomial of degree at most $r$.
Let 
\[
\Sigma = \left\{ \sigma \in  \prod_{K \in \mathcal{T}_h} H^1S_2^0(K) \mid i_{K_1,e}^*(\sigma|_{K_1}) = i_{K_2,e}^*(\sigma|_{K_2}), \forall e = K_1 \cap K_2 \in \mathcal{E}_h \right\},
\]
where $i_{K_j,e}^*$ denotes the pullback under the inclusion $i_{K_j,{e}}: e \hookrightarrow K_j$. 
In particular, we define the set  
\[  
\Sigma_{+} = \{ \sigma \in \Sigma \mid \sigma(x) \succ 0, \forall x \in \mathring{K}, K \in \mathcal{T}_h \},  
\]  
which consists of all \(\sigma \in \Sigma\) that are positive definite and serves as the set of metrics on \(M_h\).
For a tensor field \( g \in \Sigma_{+} \), its tangential-tangential components are single-valued along each edge \( e \in \mathcal{E}_h \), allowing us to define the volume form and the unit tangent vector with respect to \( g \) along \( e \) as  
\[
\mathrm{d}s_g = \sqrt{g(\tau , \tau) } \, \mathrm{d}s, \quad \text{and} \quad \tau_g = \frac{\tau}{\sqrt{g(\tau , \tau)}},
\]
where \( \mathrm{d}s \) and \( \tau \) are the volume form and unit tangent vector on \( e \) with respect to the induced Euclidean metric. Moreover, we denote $\omega_g = \sqrt{\det g} \mathrm{d}x$ the volume form on $M_h$ with respect to $g \in \Sigma_+$. 
For scalar functions $u,v$ on $M_h$, we denote 
\[ \langle u, v \rangle_{g, e} = \int_e uv \, \mathrm{d}s_g, \quad 
(u,v)_{g,M_h} = \int_{M_h} u v \sqrt{\det g}\mathrm{d}x 
= \sum_{K \in \mathcal{T}_h} \int_K u v \sqrt{\det g}\mathrm{d}x, \]
and 
\[ 
\hspace{-3.5pt}
(\nabla_g u, \nabla_g v)_{g,M_h} = 
\int_{M_h} g^{-1}\nabla u \cdot \nabla v \sqrt{\det g}\mathrm{d}x,
\quad 
(\sigma, \rho)_{g, M_h} = 
\int_{M_h} \Tr\left(g^{-1} \sigma g^{-1} \rho\right) \sqrt{\det g}\mathrm{d}x,
\]
where $\nabla u = (\partial_1 u, \partial_2 u)^{\top} \in \mathbb{R}^2$ and \( \sigma, \rho \in \Sigma \), with \( g \in \Sigma_{+} \).   
If we use the induced Euclidean metric $i_{M_h}^*\delta_{\mathbb{R}^3}$ on $M_h$, we denote
\[
(u,v)_{M_h} = \int_{M_h} u v \mathrm{d}x, \quad 
(\nabla u, \nabla v)_{M_h} = \int_{M_h} \nabla u \cdot \nabla v \mathrm{d}x, \quad 
(\sigma, \rho)_{M_h} = \int_{M_h} \Tr\left(\sigma\rho\right)\mathrm{d}x.
\]
We denote by $L^p(g, M_h)$, $1 \leq p \leq \infty$, and $H^1(g, M_h)$ the Sobolev spaces on $M_h$ associated with the metric $g \in \Sigma_+$. The Sobolev spaces with respect to the 
induced Euclidean metric $i_{M_h}^*\delta_{\mathbb{R}^3}$ are denoted as $L^p(M_h)$ and $H^1(M_h)$.

\subsection{Finite element space}  
We define the Lagrange element space of degree \( q \geq 1 \) as  
\[
V_h = \{ v \in C(M_h) \mid v|_K \in P_q(K), \, \forall K \in \mathcal{T}_h \},
\]
where \( P_q(K) \) denotes the space of polynomials on \( K \) of degree at most \( q \).  
The Regge element space of degree \( r \geq 0 \) \cite{li2018regge} is given by  
\[
\Sigma_h = \{ \sigma \in \Sigma \mid \sigma|_K \in P_r S_2^0(K), \, \forall K \in \mathcal{T}_h \}.
\]
Additionally, we define \(\Sigma_{h,+} = \Sigma_h \cap \Sigma_{+}\), 
which represents the set of all \( \sigma \in \Sigma_h \) that are positive definite.

For any $\sigma \in \Sigma$ and $g \in \Sigma_+$, we define  a metric-dependent projection $P_g: \Sigma \rightarrow \Sigma_h$ by requiring: 
\begin{subequations}\label{eq:Regge-interpolation}  
    \begin{align}
    \label{eq:Regge-interpolation-a} 
    \langle P_{g} \sigma(\tau_{g},\tau_{g}), v \rangle_{g, e} &= 
    \langle  \sigma(\tau_{g},\tau_{g}), v \rangle_{g, e}, 
    \quad \forall v \in P_r(e), e \in \mathcal{E}_h,  \\
    \label{eq:Regge-interpolation-b} 
    (P_{g} \sigma, \rho)_{g, K}& = (\sigma, \rho)_{g, K}, 
    \quad\quad\quad\quad \forall  \rho \in P_{r-1}S^0_2(K), K \in \mathcal{T}_h,
    \end{align}
\end{subequations}
where $P_r(e)$ denotes the space of polynomials on $e$ of degree at most $r$.
{If we use the induced Euclidean metric \(i_{M_h}^*\delta_{\mathbb{R}^3}\) in \eqref{eq:Regge-interpolation}, the projection defined in \eqref{eq:Regge-interpolation} coincides with the canonical Regge interpolation studied in \cite[Section 2.3.2]{li2018regge}, which we simply denote by \(P_h\).} 
Additionally, we define a metric-dependent \( L^2 \) projection \( \Pi_g: \Sigma \rightarrow \Sigma_h \) as  
\begin{equation}\label{eq:L2-projection} 
( \Pi_g \sigma, \rho )_{g, M_h} = (\sigma, \rho )_{g, M_h}, \quad \forall \rho \in \Sigma_h. 
\end{equation}

\subsection{Finite element spatial semi-discretization}
We propose the following two finite element spatial semi-discretization schemes for the coupled system \eqref{eq:Ref-Ricci}. The first scheme is based on the metric-dependent projection operator defined in \eqref{eq:Regge-interpolation}:  
Find  \( g_h(t) \in \Sigma_h \) and \( \kappa_h(t) \in V_h \) such that 
\begin{subequations}\label{eq:num-Ric}  
\begin{align}
\label{eq:num-Ric-g} 
\partial_t g_h(t) + 2P_{g_h(t)}( ( \kappa_h(t) - \bar \kappa_h) g_h(t)) &= 0,
\\ 
\label{eq:num-Ric-kappa} 
\frac{\mathrm{d}}{\mathrm{d} t} ( \kappa_h(t), \phi_h )_{g_h(t), M_h} + 
    ( \nabla_{g_h} \kappa_h(t), \nabla_{g_h} \phi_h )_{g_h(t), M_h} &= 0,
    \quad \forall \phi_h \in V_h, \partial_t \phi_h = 0.
\end{align}
\end{subequations}
The second scheme utilizes the \( L^2 \) projection defined in \eqref{eq:L2-projection}: 
Find  \( g_h(t) \in \Sigma_h \) and \( \kappa_h(t) \in V_h \) such that
\begin{subequations}\label{eq:num-Ric-L2}  
    \begin{align}
    \label{eq:num-Ric-L2-g} 
    \partial_t g_h(t) + 2\Pi_{g_h(t)}( ( \kappa_h(t) - \bar \kappa_h) g_h(t)) &= 0,
    \\ 
    \label{eq:num-Ric-L2-kappa} 
    \frac{\mathrm{d}}{\mathrm{d} t} ( \kappa_h(t), \phi_h )_{g_h(t), M_h} + 
        ( \nabla_{g_h} \kappa_h(t), \nabla_{g_h} \phi_h )_{g_h(t), M_h} &= 0,
        \quad \forall \phi_h \in V_h, \partial_t \phi_h = 0.
    \end{align}
    \end{subequations}
The initial values of $g_h$ and $\kappa_h$ in \eqref{eq:num-Ric} and \eqref{eq:num-Ric-L2} are taken as the canonical Regge interpolation $P_h g^{{(-l)}}(0) \in \Sigma_h$ and the Ritz projection $R_h \kappa(0) \in V_h$ which is defined in \eqref{eq:Ritz} below. 
In the following context, we consider the case of the normalized Ricci flow and define the constant \(\bar \kappa_h\) as  
\[
\bar \kappa_h = \frac{\int_{M_h} \kappa_h(0) \omega_{g_h(0)}}{ \int_{M_h} \omega_{g_h(0)}}.
\]  
The subsequent numerical analysis remains valid for the case \(\bar \kappa = \bar \kappa_h = 0\).

The main theoretical result of this paper is the convergence proof for the schemes \eqref{eq:num-Ric} and \eqref{eq:num-Ric-L2}, summarized in the following two theorems.
\begin{theorem}\label{thm:main} 
    Consider the spatial semi-discretization scheme \eqref{eq:num-Ric} for the coupled Ricci flow equations \eqref{eq:Ref-Ricci}, using the Regge element space of degree $r \geq 0$ and the Lagrange element space of degree $q \geq 1$. Suppose that the Ricci flow equation admits an exact solution $(g, \kappa)$ that is sufficiently smooth on the time interval $[0, T]$, and that the symmetric $(0,2)$-tensor field $g(t)$ is uniformly positive definite on the time interval $[0, T]$.
    Then, there exists a constant $h_0 > 0$  such that the finite element solutions given by \eqref{eq:num-Ric} satisfy the following error bound for $0<h\leq h_0${\rm:}
\[
    \| g_h(t) -  g^{(-l)}(t) \|_{L^p(M_h)} + \| \kappa_h(t) - \kappa^{(-l)}(t) \|_{L^2(M_h)} \lesssim (\ln(\frac{1}{h}))^{\bar q} h^{q+1} + \ln(\frac{1}{h})h^{r+1}, \quad t \in [0,T], 
\]
where  $p > 2$  and  $\bar{q} = {1}$ when $q = 1$, and $\bar{q} = 0$ otherwise. 
\end{theorem} 
\begin{theorem}\label{thm:main-L2} 
    Under the conditions of Theorem \ref{thm:main}, when \( r \geq 1 \) and \( q \geq 1 \), for both schemes \eqref{eq:num-Ric} and \eqref{eq:num-Ric-L2}, there exists a constant \( h_0 > 0 \) such that the finite element solutions satisfy the following error bound for \( 0 < h \leq h_0 \)\rm{:}
    \[
        \| g_h(t) -  g^{(-l)}(t) \|_{L^2(M_h)} + \| \kappa_h(t) - \kappa^{(-l)}(t) \|_{L^2(M_h)} \lesssim h^{q+1} + h^{r+1}, \quad t \in [0,T]. 
    \]
\end{theorem}
The proof of Theorem \ref{thm:main} is presented in Section \ref{sec:stab} and Section \ref{sec:err} in detail. 
Since Theorem \ref{thm:main-L2} can be proved using a very similar approach, we only provide several remarks to clarify the key differences in the proof (see Remark \ref{rmk:L2-bound}, \ref{rmk:L2-stab-e_g}, \ref{rmk:L2-stab-e_k}).
In addition, in Theorems \ref{thm:main} and \ref{thm:main-L2}, we have proven the convergence of scheme \eqref{eq:num-Ric} for arbitrary polynomial order, whereas for scheme \eqref{eq:num-Ric-L2}, the condition \( r \geq 1 \) and \( q \geq 1 \) is required.
This is because the proof of Theorem \ref{thm:main} relies on the \( L^{\infty} \) stability of the projection operator, which has so far only been shown for \( P_{g_h(t)} \) in scheme \eqref{eq:num-Ric} (see Lemma \ref{lm:stab-Pgh-Pgh*} and Appendix \ref{sec:P-Stab}).

Besides their convergence, another key feature of the proposed schemes is their ability to preserve certain geometric structures of the Ricci flow at a discrete level, including the Gauss-Bonnet theorem and area conservation. 
\begin{theorem}[Discrete Gauss--Bonnet and area conservation]\label{thm:GeoStruc}
    Let \(g_h(t)\) and \(\kappa_h(t)\) be the solutions of \eqref{eq:num-Ric} or \eqref{eq:num-Ric-L2}. Then the following identity holds:
    \begin{subequations}
        \begin{align}
            \label{eq:disGauBon}
            \int_{M_h} \kappa_h(t) \, \omega_{g_h(t)} &= \int_{M_h} \kappa_h(0) \, \omega_{g_h(0)}, \quad \forall t \in [0,T].\\
    \intertext{Moreover, for solutions of the scheme \eqref{eq:num-Ric-L2}, we have area conservation:}
    \label{eq:areaCons} 
    \int_{M_h} \omega_{g_h(t)} &= \int_{M_h} \omega_{g_h(0)}, ~\quad\quad\quad \forall t \in [0,T].
    \end{align}
    \end{subequations}    
\end{theorem}

\begin{proof}
    By choosing the test function \( \phi_h \equiv 1 \) in \eqref{eq:num-Ric-kappa} or \eqref{eq:num-Ric-L2-kappa}, we obtain  
    \[ 
        \frac{\mathrm{d}}{\mathrm{d} t} \int_{M_h} \kappa_h(t) \omega_{g_h(t)}
    =  \frac{\mathrm{d}}{\mathrm{d} t} ( \kappa_h(t), 1 )_{g_h(t), M_h} 
    = -( \nabla_{g_h} (\kappa_h(t)), \nabla_{g_h} 1 )_{g_h(t), M_h} = 0.
    \]  
    This concludes the proof of \eqref{eq:disGauBon}.
    Next, to prove \eqref{eq:areaCons}, we define
    \[
    f(t) = \int_{M_h} \omega_{g_h(t)} - \int_{M_h} \omega_{g_h(0)}.
    \]
    Taking the time derivative, we have 
    \[ \frac{\mathrm{d}}{\mathrm{d}t} f(t) = \frac{\mathrm{d}}{\mathrm{d} t} \int_{M_h}\omega_{g_h(t)} 
    = \frac{1}{2}\int_{M_h} \Tr_{g_h(t)}(\partial_t g_h(t)) \omega_{g_h(t)} 
    = \frac{1}{2}(g_h(t), \partial_t g_h(t))_{g_h(t), M_h}. \]
    Using \eqref{eq:num-Ric-L2-g} and the definition of \( \Pi_{g_h(t)} \), we obtain  
    \[
    \begin{aligned}
        \frac{1}{2}(g_h(t), &\partial_t g_h(t))_{g_h(t), M_h} = 
        (g_h(t), (\bar \kappa_h - \kappa_h(t)) g_h(t))_{g_h(t), M_h} = 
        2\int_{M_h}  (\bar \kappa_h - \kappa_h(t)) \omega_{g_h(t)} \\ 
        & = 2\big(\bar \kappa_h \int_{M_h} \omega_{g_h(t)} 
        - \int_{M_h} \kappa_h(t) \omega_{g_h(t)}\big)
        = 2 \bar \kappa_h \big(\int_{M_h} \omega_{g_h(t)} - \int_{M_h} \omega_{g_h(0)}\big),  
    \end{aligned}
    \]
    where in the last step, we apply \eqref{eq:disGauBon}.
    Then, it holds that 
    \[ \frac{\mathrm{d}}{\mathrm{d}t} f(t) = 2 \bar \kappa_h f(t). \]
    Since $f(0) = 0$, it follows that $f(t) = 0$ for every $t \in [0,T]$.    
\end{proof}

The finite element discretizations \eqref{eq:num-Ric} and \eqref{eq:num-Ric-L2} can be expressed in matrix-vector form \cite{kovacs2017convergence}. Specifically, for the finite element solutions \( g_h(t) \in \Sigma_h \) and \( \kappa_h(t) \in V_h \), we collect the values of their degrees of freedom (DOFs) as vectors \( \g(t) \in \mathbb{R}^{\dim \Sigma_h} \) and \( \ka(t) \in \mathbb{R}^{\dim V_h} \), which serve as the vector representations of the finite element solutions \( g_h(t) \) and \( \kappa_h(t) \), respectively.
We define the metric-dependent mass matrix $\M[ \g(t)] \in \mathbb{R}^{{\rm dim} V_h \times {\rm dim} V_h }$ and stiffness matrix $\A[ \g(t)] \in \mathbb{R}^{{\rm dim} V_h \times {\rm dim} V_h }$ as: for $i,j = 1,\dots, {\rm dim}V_h,$
\[ (\M[ \g(t)])_{i,j} = (\phi_i, \phi_j)_{g_h(t), M_h}, \quad 
(\A[ \g(t)])_{i,j} = (\nabla_{g_h(t)} \phi_i, \nabla_{g_h(t)} \phi_j )_{g_h(t), M_h}, \quad  
   \]
with the finite element nodal basis functions $\phi_i \in V_h$.
Therefore, algebraically scheme \eqref{eq:num-Ric} corresponds to a differential algebraic equation (DAE)
\begin{subequations}\label{eq:DAE}  
\begin{align}
\label{eq:DAE-g} 
\dot{\g}(t) + f[\g(t), \ka(t)] &= 0, \\ 
\label{eq:DAE-kappa} 
\frac{\mathrm{d}}{\mathrm{d} t}\big( \M[ \g(t)] {\ka}(t) \big) +  
\A[ \g(t) ]  \ka(t) &= 0, \\ 
\intertext{where \( f[\g(t), \ka(t)] \in \mathbb{R}^{{\rm dim} \Sigma_h} \) collects the DOFs associated with \( 2P_{g_h(t)}( ( \kappa_h(t) - \bar \kappa_h) g_h(t)) \) in \eqref{eq:num-Ric-g}. Similarly, \eqref{eq:num-Ric-L2} also corresponds to a differential algebraic equation, and it can be obtained by replacing \eqref{eq:DAE-g} with 
} 
\label{eq:DAE-g-L2} 
\dot{\g}(t) + \tilde  f[\g(t), \ka(t)] &= 0, 
\end{align}
\end{subequations}
where  \( \tilde f[\g(t), \ka(t)] \in \mathbb{R}^{{\rm dim} \Sigma_h} \) collects the DOFs associated with \( 2\Pi_{g_h(t)}( ( \kappa_h(t) - \bar \kappa_h) g_h(t)) \) in \eqref{eq:num-Ric-L2-g}.

\subsection{Approximations of the Riemannian metric and the Gaussian curvature}
Let \( g(t) \) be the smooth solution of equation \eqref{eq:Ref-Ricci}. We denote its pullback on \( M_h \) and the canonical Regge interpolation as
\[ g^{(-l)}(t) = a^*g(t) \in \Sigma_+, \quad \text{and} \quad 
g_h^*(t) = P_hg^{(-l)}(t) \in \Sigma_h. \]
{The following property follows directly from the definition of the canonical Regge interpolation and its approximation properties; see Section~2.3 in~\cite{li2018regge}. For \(t\in[0,T]\),} 
\begin{subequations}\label{eq:prop-g_h^*}  
\begin{align}
\label{eq:prop-g_h^*-a} 
    \partial_t g_h^*(t) = \partial_t P_h g^{(-l)}(t)&
    = P_h \partial_t g^{(-l)}(t)
    = P_h (\partial_t g)^{(-l)}(t), \\ 
\label{eq:prop-g_h^*-b} 
 \| g_h^*(t) - g^{(-l)}(t) \|_{W^{k,p}(K)} & \lesssim h^{r+1-k} \| g^{(-l)}(t) \|_{W^{r+1,p}(K)},\\ 
\label{eq:prop-g_h^*-c}
    \| \partial_t g_h^*(t) - (\partial_t g)^{(-l)}(t) \|_{W^{k,p}(K)} & \lesssim h^{r+1-k} \| (\partial_t g)^{(-l)}(t) \|_{W^{r+1,p}(K)}, 
\end{align}
\end{subequations}
{where \(k\in\{0,\dots,r+1\}\), \(p\in[1,\infty]\), and \(K\in\mathcal{T}_h\).
The following result, established in Lemmas~4.5--4.6 of~\cite{gawlik2020high}, is useful for controlling errors in the inverse and determinant of a metric tensor.} 
For a triangle \(K\in\mathcal{T}_h\), let \(g\) and \(\tilde g\) be smooth symmetric \((0,2)\)-tensor fields on \(K\).
The following estimates hold for \(p\in[1,\infty]\):
\begin{subequations}\label{eq:inv-det}  
\begin{align}
\label{eq:inv-det-a} 
        \| (\tilde{g})^{-1} -  g^{-1} \|_{L^p(K)} &\leq 
        \frac{\|g^{-1}\|_{L^\infty(K)}^2}{1 - 
        \| \tilde{g} -  g  \|_{L^\infty(K)}\|g^{-1}\|_{L^\infty(K)}} \, \| \tilde{g} - g \|_{L^p(K)}, \\ 
\label{eq:inv-det-b} 
\hspace{-2mm}    
        \| \sqrt{\det (\tilde{g})} -  \sqrt{\det (g)} \|_{L^p(K)} &\leq 
        \bigl(\|\tilde{g}\|_{L^\infty(K)} + \|g\|_{L^\infty(K)}\bigr)\, \| g^{-1}\|_{L^\infty(K)} \, \| \tilde{g} - g \|_{L^p(K)}.
\end{align}
\end{subequations}
Under the conditions of Theorem \ref{thm:main}, a direct consequence of \eqref{eq:inv-det} and \eqref{eq:prop-g_h^*} is that, if the mesh size \( h \) is sufficiently small, the canonical Regge interpolant \( g_h^*(t) \in \Sigma_h \) is uniformly positive definite and satisfies
\begin{equation}\label{eq:det-inv-app} 
{   \| \sqrt{\det g_h^*(t)} - \sqrt{\det g^{(-l)}(t)} \|_{L^{\infty}(M_h)} + 
    \|  (g_h^*)^{-1}(t) - (g^{(-l)})^{-1}(t) \|_{L^{\infty}(M_h)} \lesssim h^{r+1},} 
\end{equation}
for $t \in [0,T]$. In the following, we always assume that \( h \) is sufficiently small such that \( g_h^*(t) \in \Sigma_{h,+} \), which implies that \( g_h^*(t) \) can serve as a Riemannian metric on \( M_h \).

\begin{lemma}[Norm equivalence of $g_h^*$]\label{lm:norm-equ}
Under the conditions of Theorem \ref{thm:main}, there exists a constant \( h_0 > 0 \) such that for $0<h\leq h_0$ and \( v, w \in H^1(\mathcal{M}) \), the following estimates hold for $t \in [0,T]$,
\begin{subequations}
    \begin{align}
        \label{eq:M-g-g*} 
&\quad | (\nabla_g v,\nabla_g w)_{g(t), \mathcal{M}} - (\nabla_{g_h^*} v^{(-l)}, \nabla_{g_h^*} w^{(-l)})_{g_h^*(t), M_h} | \\ \notag
&\lesssim h^{r+1}\| \nabla  v^{(-l)}\|_{L^2(M_h)} \| \nabla  w^{(-l)}\|_{L^2(M_h)},  
 \\ 
\label{eq:A-g-g*} 
&| ( v, w)_{g(t), \mathcal{M}} - ( v^{(-l)}, w^{(-l)})_{g_h^*(t), M_h} | \lesssim h^{r+1}\| v^{(-l)}\|_{L^2(M_h)} \| w^{(-l)}\|_{L^2(M_h)}. \\ 
\intertext{Moreover, the following norm equivalence holds for $t \in [0,T]$,}
 \label{eq:g-g*-equi-L2} 
&\|v^{(-l)}\|_{L^2(M_h)} \eqsim
\|v^{(-l)}\|_{L^2(g_h^*(t), M_h)} \eqsim
\|v\|_{L^2(g(t), \mathcal{M})}\eqsim
\|v\|_{L^2(\mathcal{M})}, \\ 
\label{eq:g-g*-equi-H1}
&\| \nabla v^{(-l)} \|_{L^2(M_h)}  \eqsim
\| \nabla_{g_h^*} v^{(-l)} \|_{L^2(g_h^*(t), M_h)} \eqsim
\| \nabla_{g} v \|_{L^2(g(t), \mathcal{M})} \eqsim
\| \nabla v \|_{L^2(\mathcal{M})}.
    \end{align}
\end{subequations}
\end{lemma}
\begin{proof}
For a flat triangle \( K \in \mathcal{T}_h \), let \( \tilde{K} = a(K) \subset \mathcal{M} \) be a curved triangle that lies on the smooth manifold \( \mathcal{M} \). 
Then, the set 
$\tilde{\mathcal{T}_h} = \{ \tilde{K} = a(K), \text{ where } K \in \mathcal{T}_h \},$ 
forms a curved triangulation of the smooth manifold \( \mathcal{M} \). 
Using the property of the pullback operator, the following expressions hold:
\[
\begin{aligned}
(v,w)_{g(t), \tilde K}
&= (v^{(-l)},w^{(-l)})_{g^{(-l)}, K} 
= \int_K  v^{(-l)} w^{(-l)} \sqrt{\det g^{(-l)}}\mathrm{d}x, \\
 \hspace{-31.9pt}
 (\nabla_g v,\nabla_g w)_{g(t), \tilde K}
&= \int_K  (g^{(-l)})^{-1} \nabla v^{(-l)} \cdot \nabla w^{(-l)} \sqrt{\det g^{(-l)}}\mathrm{d}x.
\end{aligned}
\]
Similarly, the following expressions hold:
\[
\begin{aligned}
    (v^{(-l)},w^{(-l)})_{g_h^*(t), K} &=  \int_K  v^{(-l)} w^{(-l)} \sqrt{\det g_h^*}\mathrm{d}x, \\ 
(\nabla_{g_h^*} v^{(-l)}, \nabla_{g_h^*} w^{(-l)})_{g_h^*(t), K} & = 
\int_K  (g_h^*)^{-1} \nabla v^{(-l)} \cdot \nabla w^{(-l)} \sqrt{\det g_h^*}\mathrm{d}x.
\end{aligned}
\]
By comparing the above expressions, summing over all \(K\in\mathcal{T}_h\), and invoking \eqref{eq:det-inv-app}, we obtain \eqref{eq:M-g-g*} and \eqref{eq:A-g-g*}. Under the conditions of Theorem \ref{thm:main}, both $g^{(-l)}(t)$ and its inverse $(g^{(-l)}(t))^{-1}$ are uniformly bounded on the time interval $[0,T]$. According to \eqref{eq:prop-g_h^*} and \eqref{eq:det-inv-app}, for sufficiently small $h$, both $g_h^*(t)$ and its inverse $(g_h^*(t))^{-1}$ are also uniformly bounded on the time interval $[0,T]$. Therefore, \eqref{eq:g-g*-equi-L2} and \eqref{eq:g-g*-equi-H1} follow from taking $v = w$ in the above expressions and applying a H\"older inequality.
\end{proof}

For a smooth scalar function \( u(t) = u(\cdot, t): \mathcal{M} \to \mathbb{R} \), we define its Ritz projection onto the Lagrange finite element space \( V_h \) as \( R_h u(t) \in V_h \), by requiring 
\begin{subequations} \label{eq:Ritz} 
    \begin{align}
      \label{eq:Ritz-a}
     (\nabla_{g_h^*} R_h u(t), \nabla_{g_h^*}  \phi_h  )_{g_h^*, M_h} &= (\nabla_{g} u(t), \nabla_{g}\phi_h^{(l)}  )_{g, \mathcal{M}}, \quad \forall \phi_h \in V_h, \\ 
      \label{eq:Ritz-b}
     ((R_h u(t))^{(l)}, 1)_{\mathcal{M}} &= (u(t), 1)_{\mathcal{M}}.
    \end{align}  
\end{subequations}
The following lemma provides an error estimate for the Ritz projection and its time derivative. 
\begin{lemma}[Properties of Ritz projection]\label{lm:app-Ritz} 
For a smooth function \( u(t) = u(\cdot, t): \mathcal{M} \to \mathbb{R} \), let \( R_h u(t) \in V_h \) be its Ritz projection onto the Lagrange finite element space \( V_h \) of degree \( q \), as defined in \eqref{eq:Ritz}, 
where \( g_h^* = P_h g^{(-l)}(t) \in \Sigma_{h,+} \) is the canonical Regge interpolation onto the Regge element space of degree \( r \). 
Then, under the conditions of Theorem \ref{thm:main}, there exists a constant \( h_0 > 0 \) such that for $0<h\leq h_0$, the following error estimates hold for $t \in [0,T]$:
\begin{subequations}\label{eq:app-Ritz}
    \begin{align}
        \label{eq:app-Ritz-H1}
       & \| (u(t))^{(-l)} - R_h u(t) \|_{H^1(M_h)} + 
        \|\partial_t (u(t))^{(-l)} - \partial_t R_h u(t)\|_{H^1(M_h)} 
        \lesssim h^{r+1} + h^q, 
         \\
      \label{eq:app-Ritz-L2}
     & \| (u(t))^{(-l)} - R_h u(t) \|_{L^2(M_h)} + 
      \|\partial_t (u(t))^{(-l)} - \partial_t R_h u(t)\|_{L^2(M_h)} 
        \lesssim h^{r+1} + h^{q+1}, \\  
\label{eq:app-Ritz-Linf} 
   & \| (u(t))^{(-l)} - R_h u(t) \|_{L^{\infty}(M_h)} 
      \lesssim \ln(\frac{1}{h}) h^{r+1} + (\ln(\frac{1}{h}))^{\bar q} h^{{ q+1}},
\end{align}
\end{subequations}
where  $\bar{q} = {1}$ when $q = 1$, and $\bar{q} = 0$ otherwise. 
\end{lemma} 
{\begin{proof}
The \(L^2\), \(H^1\), and \(L^\infty\) error estimates for the Ritz projection \(R_h u(t)\) follow the framework developed in \cite[Section 3]{demlow2009higher}.
Specifically, we first rewrite \eqref{eq:Ritz-a} in the equivalent form
\begin{equation}\label{eq:Ritz-1-a-ref}
(\nabla_{g} ((R_h u)^{(l)} - u), \nabla_{g} \phi_h^{(l)})_{g, \mathcal{M}}
= F(\phi_h^{(l)}), \quad \forall \phi_h \in V_h,
\end{equation}
where the linear functional \(F\) is defined by
\[
F(\phi_h^{(l)}) =
(\nabla_{g} (R_h u)^{(l)}, \nabla_{g} \phi_h^{(l)})_{g, \mathcal{M}}
- (\nabla_{g_h^*} R_h u(t), \nabla_{g_h^*} \phi_h)_{g_h^*, M_h}.
\]
Then, one can verify that assumptions A1-A4 in \cite[Section 3.1]{demlow2009higher} are satisfied in our setting. 
Therefore, by \cite[Theorem 3.1, (3.3) and (3.5)]{demlow2009higher} and \cite[Theorem 3.2, (3.14)]{demlow2009higher}, together with Lemma \ref{lm:norm-equ}, the following \(L^2\), \(H^1\), and \(L^\infty\) estimates for \(R_h u(t)\) hold:
\begin{subequations}\label{eq:L2H1Linf-Ritz}  
\begin{align}
\label{eq:H1-Ritz} 
\| u^{(-l)} - R_h u \|_{H^1(M_h)} &\lesssim h^{r+1} + h^q,  \\ 
\label{eq:L2-Ritz} 
\| u^{(-l)} - R_h u \|_{L^2(M_h)} &\lesssim h^{r+1} + h^{q+1}, \\
\label{eq:L2-inf} 
\| u^{(-l)} - R_h u \|_{L^{\infty}(M_h)}
&\lesssim \ln\!\bigl(\frac{1}{h}\bigr)\, h^{r+1}
+ \bigl(\ln\!\bigl(\frac{1}{h}\bigr)\bigr)^{\bar q} h^{q+1}.
\end{align}
\end{subequations}
The error estimates for the time derivative \(\partial_t R_h u(t)\) are deferred to Appendix \ref{sec:Ritz}.
\end{proof}
} 

Let $\kappa(t): \mathcal{M} \rightarrow \mathbb{R}$ be the smooth solution of \eqref{eq:Ref-Ricci}.  In the following, we define 
\[ \kappa_h^*(t) = R_h \kappa(t) \in V_h. \]
Based on Lemma \ref{lm:app-Ritz} (Properties of Ritz projection), \eqref{eq:prop-g_h^*}, and \eqref{eq:det-inv-app}, the following boundedness results hold for $r \geq 0$ and $q \geq 1$:
\begin{subequations}\label{eq:bound-g*-k*} 
    \begin{align}
    \label{eq:bound-k*} 
    \|\kappa_h^* \|_{L^{\infty}(0,T;W^{1,\infty}(M_h))} + 
    \|\partial_t\kappa_h^*\|_{L^{\infty}(0,T; L^{\infty}(M_h))} & \leq C, \\
    \label{eq:bound-g*}
    \|g_h^* \|_{L^{\infty}(0,T; L^{\infty}(M_h))} + 
    \| (g_h^*)^{-1} \|_{L^{\infty}(0,T; L^{\infty}(M_h))} &+
    \\ \notag 
    \| \sqrt{\det g_h^*} \|_{L^{\infty}(0,T; L^{\infty}(M_h))} + 
    \| \partial_t g_h^* \|_{L^{\infty}(0,T; L^{\infty}(M_h))} &\leq  
    C.
    \end{align}
\end{subequations}
Both \eqref{eq:bound-k*} and \eqref{eq:bound-g*} will be frequently employed in the subsequent stability analysis to control the nonlinearity.

\section{Stability estimate}\label{sec:stab}

\subsection{{Consistency defects}  and error equations}
Recall that \( g_h^*(t) \in \Sigma_h \) and \( \kappa_h^*(t) \in V_h \) are the canonical Regge interpolation of \( g(t) \), and the Ritz projection of \( \kappa(t) \), respectively, which serve as approximations to the smooth solutions of equation \eqref{eq:Ref-Ricci}.
We use $e_g(t) = g_h(t) - g_h^*(t) \in \Sigma_h$ and $e_{\kappa}(t) = \kappa_h(t) - \kappa_h^*(t) \in V_h$ to denote the finite element error functions.  
The approximations $ g_h^*(t)$ and $\kappa_h^*(t)$ satisfy the numerical schemes \eqref{eq:num-Ric} and \eqref{eq:num-Ric-L2} up to { the consistency defect functions}  $ d_g(t), d_{L^2,g}(t) \in \Sigma_h$ and $ d_\kappa(t) \in V_h$, i.e.,  
\begin{subequations}\label{eq:defect-fun}  
\begin{align}
\label{eq:defect-g} 
\partial_t g_h^*(t) + 2P_{g_h^*(t)}( ( \kappa_h^*(t) - \bar \kappa_h) g_h^*(t)) &= d_g(t), \\ 
\label{eq:defect-g-L2} 
\partial_t g_h^*(t) + 2\Pi_{g_h^*(t)}( ( \kappa_h^*(t) - \bar \kappa_h) g_h^*(t)) &= d_{L^2,g}(t), \\ 
\label{eq:defect-kappa} 
\hspace{-8mm}
\frac{\mathrm{d}}{\mathrm{d} t} ( \kappa_h^*(t), \phi_h )_{g_h^*(t), M_h} + 
( \nabla_{g_h^*} \kappa_h^*(t), \nabla_{g_h^*} \phi_h )_{g_h^*(t), M_h} &= 
( d_{\kappa}(t), \phi_h )_{g_h^*(t), M_h}, 
\end{align}
\end{subequations}
{for all $\phi_h \in V_h$. } 
Let $\gs(t) \in \mathbb{R}^{{\rm dim} \Sigma_h}$ and $\kas(t), \ekappa(t),\dkappa(t) \in \mathbb{R}^{{\rm dim} V_h}$ be the vector representations of $g_h^*(t) \in \Sigma_h$ and $\kappa_h^*(t), e_{\kappa}(t), d_\kappa(t) \in V_h$, respectively. 
The matrix-vector formulation of the {consistency defect}  equation \eqref{eq:defect-kappa} is given by: 
\[
    \frac{\mathrm{d}}{\mathrm{d} t} \big( \M[\gs(t)] \ka^*(t) \big) + \A[\gs(t)] \ka^*(t) = 
    \M[\gs(t)] \dkappa(t).
\]
 Substituting this into \eqref{eq:DAE-kappa}, we obtain the error equation of $e_\kappa$:
\begin{equation}\label{eq:error-kappa} 
\begin{aligned}
    \frac{\mathrm{d}}{\mathrm{d} t}\big( \M[ \g(t)] {\bf e_\kappa}(t) \big) + 
    \A[  \g(t) ] {\bf e_\kappa}(t) &= 
    \frac{\mathrm{d}}{\mathrm{d} t}\big( ( \M[\gs(t)] - \M[\g(t)] )\kas(t) \big)  \\ 
    &+ \big( \A[\gs(t)] - \A[\g(t)] \big){\ka^*}(t)  - \M[\gs(t)] \dkappa.     
\end{aligned}
\end{equation}
Similarly,  subtracting \eqref{eq:defect-g} from \eqref{eq:num-Ric-g} and \eqref{eq:defect-g-L2} from \eqref{eq:num-Ric-L2-g},
the error equations of $e_g$ for schemes \eqref{eq:num-Ric} and \eqref{eq:num-Ric-L2} read 
\begin{subequations}
\begin{align}
    \label{eq:error-g} 
        \hspace{-5mm}
    \partial_t e_g + &2P_{g_h}(e_{\kappa} g_h) 
    + 2P_{g_h}((\kappa_h^* - \bar \kappa_h)e_g)
    + 2(P_{g_h} - P_{g_h^*})( ( \kappa_h^* - \bar \kappa_h) g_h^*) + d_g = 0, \\
    \label{eq:error-g-L2} 
        \hspace{-5mm}
    \partial_t e_g + &2\Pi_{g_h}(e_{\kappa} g_h) 
    + 2\Pi_{g_h}((\kappa_h^* - \bar \kappa_h)e_g)
    + 2(\Pi_{g_h} - {\Pi}_{g_h^*})( ( \kappa_h^* - \bar \kappa_h) g_h^*) + d_{L^2, g} = 0.
\end{align}
\end{subequations}
The stability analysis will be based on the error equations in \eqref{eq:error-g}, \eqref{eq:error-g-L2} and \eqref{eq:error-kappa}.

\subsection{Preparation for stability estimate}

Let \( p > 2 \) be a fixed number. We first consider the \( L^p \) stability estimate for the scheme \eqref{eq:num-Ric} with arbitrary polynomial orders \( r \geq 0 \) and \( q \geq 1 \).
Let \( t^* \in (0, T] \) be the supremum of time such that the following estimate holds (with a coefficient of 1 on the right-hand side):
\begin{equation}\label{eq:bound-e_g-e_k} 
\| e_g \|_{L^{\infty}(0,t^*;L^p(M_h))} + 
\| e_\kappa \|_{L^{\infty}(0,t^*;L^2(M_h))} \leq 
h^{\frac{1}{2} + \frac{1}{p}}.
\end{equation} 
Since \( e_g(0) = e_\kappa(0) = 0 \) and the semidiscrete finite element solutions are continuous in time, it follows that \( t^* > 0 \). Ultimately, our proof will be completed by showing that \( t^* = T \).
The bound in \eqref{eq:bound-e_g-e_k} ensures the boundedness of the numerical solutions of \eqref{eq:num-Ric} on the time interval \( [0, t^*] \). For instance, by utilizing the inverse estimate and the boundedness in \eqref{eq:bound-g*}, we obtain  
\[ \| g_h \|_{L^{\infty}(0,t^*;L^{\infty}(M_h))} \lesssim 
\| g_h^* \|_{L^{\infty}(0,t^*;L^{\infty}(M_h))} + 
\| e_g \|_{L^{\infty}(0,t^*;L^p(M_h))} h^{-\frac{2}{p}} \lesssim
1 + h^{\frac{1}{2}-\frac{1}{p}}.
\]
Since $ \frac{1}{2}-\frac{1}{p}> 0 $, combining \eqref{eq:inv-det}, for sufficiently small $h$, we obtain the boundedness of $g_h(t)$, 
\begin{equation}\label{eq:bound-gh} 
    \| g_h \|_{L^{\infty}(0,t^*;L^{\infty}(M_h))} + 
    \| g_h^{-1} \|_{L^{\infty}(0,t^*;L^{\infty}(M_h))} + 
    \| \sqrt{\det{g_h}} \|_{L^{\infty}(0,t^*;L^{\infty}(M_h))} \leq C.
\end{equation} 
Moreover, for another fixed number $\tilde q$ satisfying $2<\tilde q<\frac{4p}{p-2}$, using an inverse estimate and boundedness in \eqref{eq:bound-k*}, it holds that
\[ \| \kappa_h \|_{L^{\infty}(0,t^*;L^{\tilde q}(M_h))} \lesssim 
\| \kappa_h^* \|_{L^{\infty}(0,t^*;L^{\tilde q}(M_h))} + 
\| e_{\kappa} \|_{L^{\infty}(0,t^*;L^{2}(M_h))}h^{-1 + \frac{2}{\tilde q}} \lesssim 
1 + h^{ \frac{1}{p} - \frac{1}{2} +  \frac{2}{\tilde q} }.  \]
Since $\frac{1}{p} - \frac{1}{2} +  \frac{2}{\tilde q} = \frac{4p +2{\tilde q} - p{\tilde q} }{2p{\tilde q}} > 0 $, 
for sufficiently small $h$, we obtain the boundedness of $\kappa_h(t)$,  
\begin{equation}\label{eq:bound-kh} 
    \| \kappa_h \|_{L^{\infty}(0,t^*;L^{\tilde q}(M_h))} \leq C, \quad \text{ where }
    2 < \tilde q < \frac{4p}{p-2}.
\end{equation} 
By a similar argument as in Lemma \ref{lm:norm-equ} (Norm equivalence of $g_h^*$), a direct application of the boundedness results in \eqref{eq:bound-g*} and \eqref{eq:bound-gh} lead to the following lemma.
\begin{lemma}[Norm equivalence of $g_h$ ]\label{lm:norm-equi-gh} 
Under the boundedness results in \eqref{eq:bound-g*} and \eqref{eq:bound-gh},  it holds that for $t \in [0,t^*]$ and $s \in [2, \infty]$,  
\begin{subequations}\label{eq:gh-g*-equi}
    \begin{align}
        \label{eq:gh-g*-equi-Ls} 
        \|v^{(l)}\|_{L^s(\Man)} \eqsim \|v\|_{L^s(M_h)} &\eqsim \|v\|_{L^s(g_h^*, M_h)} \eqsim \|v\|_{L^s(g_h, M_h)}, \quad 
        &&\forall  v \in L^s(M_h), \\ 
        \label{eq:gh-g*-equi-H1}
        \hspace{-2mm}
        \|\nabla v^{(l)}\|_{L^2(\Man)} \eqsim \| \nabla v \|_{L^2(M_h)} &\eqsim \| \nabla_{g_h^*} v \|_{L^2(g_h^*, M_h)} \eqsim \| \nabla_{g_h} v \|_{L^2(g_h, M_h)}, \quad 
        &&\forall  v \in H^1(M_h).\\ 
    \intertext{Moreover, for tensor fields $\sigma, \rho \in \Sigma$, it holds that for $t \in [0,t^*]$, }
    \label{eq:gh-g*-inner-tensor}    
    |(\sigma, \rho)_{g_h^*,M_h} -  (\sigma, \rho)_{g_h,M_h} | &\lesssim 
    \|g_h^* - g_h\|_{L^2(M_h)} \|\sigma\|_{L^{\infty}(M_h)}\|\rho\|_{L^2(M_h)}, 
    \\ \label{eq:gh-g*-equi-tensor}
     \| \sigma \|_{L^2(M_h)} &\eqsim \| \sigma \|_{L^2(g_h^*(t), M_h)} \eqsim \| \sigma\|_{L^2(g_h(t), M_h)}.
    \end{align}
\end{subequations}
\end{lemma} 
\begin{proof}
{    
Under the boundedness results in \eqref{eq:bound-g*} and \eqref{eq:bound-gh}, 
the norm equivalence results in \eqref{eq:gh-g*-equi-Ls}, \eqref{eq:gh-g*-equi-H1}, and \eqref{eq:gh-g*-equi-tensor} can be proven 
by a similar argument as in Lemma \ref{lm:norm-equ} (Norm equivalence of $g_h^*$) and are therefore omitted.
For the proof of \eqref{eq:gh-g*-inner-tensor}, we have
    \[ 
    \begin{aligned}
        (\sigma, \rho)_{g_h^*, M_h} &= \int_{M_h} {\rm Tr}((g_h^*)^{-1}\sigma (g_h^*)^{-1}\rho) \sqrt{\det g_h^*}\mathrm{d}x,
    \\
    (\sigma, \rho)_{g_h, M_h} &= \int_{M_h} {\rm Tr}((g_h)^{-1}\sigma (g_h)^{-1}\rho) \sqrt{\det g_h}\mathrm{d}x. 
    \end{aligned}
    \]
Comparing the above expressions and invoking the boundedness results in \eqref{eq:bound-g*} and \eqref{eq:bound-gh}, together with \eqref{eq:inv-det}, we obtain, for \(t\in[0,t^*]\),
\[ 
\begin{aligned}
    & \quad|(\sigma, \rho)_{g_h^*(t), M_h} - (\sigma, \rho)_{g_h(t), M_h}| \\ 
        &\lesssim 
    \|\sigma\|_{L^\infty(M_h)} \|\rho\|_{L^2(M_h)} 
    \big( \|\sqrt{\det g_h^*} - \sqrt{\det g_h}\|_{L^2(M_h)} + 
    \|(g_h^*)^{-1} - (g_h)^{-1}\|_{L^2(M_h)}
    \big) \\ 
    & \lesssim 
    \|\sigma\|_{L^\infty(M_h)} \|\rho\|_{L^2(M_h)} \| g_h^* - g_h \|_{L^2(M_h)}. 
\end{aligned}
\]
}
\end{proof} 

Furthermore,  we can establish the following stability and approximation results for the projection operators defined in \eqref{eq:Regge-interpolation} and \eqref{eq:L2-projection}. The proof of these two lemmas are postponed to Appendix \ref{sec:P-Stab}. 
\begin{lemma}[Stability of $P_{g_h}$ and $P_{g_h^*}$]\label{lm:stab-Pgh-Pgh*} 
{Under the boundedness results in \eqref{eq:bound-g*} and \eqref{eq:bound-gh}, for any  triangle \(K \in \mathcal{T}_h\), let $\sigma$ be a smooth symmetric $(0,2)$-tensor field on $K$.  Then it holds that 
\begin{subequations}
\begin{align}
    \label{eq:stab-inf}
    & \| P_{g_h(t)} \sigma \|_{L^{\infty}(K)} \lesssim \| \sigma \|_{L^{\infty}(K)} , \quad t \in [0,t^*],
    \quad 
    \| P_{g_h^*(t)} \sigma \|_{L^{\infty}(K)} \lesssim \| \sigma \|_{L^{\infty}(K)} , \quad t \in [0,T]. \\ 
\intertext{Moreover, if $\sigma$ is a polynomial symmetric $(0,2)$-tensor field on $K$, it holds that, for $2 \leq s \leq +\infty$,}
    \label{eq:stab-Ls}
    & \| P_{g_h(t)} \sigma \|_{L^{s}(K)} \lesssim \| \sigma \|_{L^{s}(K)} , \quad t \in [0,t^*],
    \quad 
    \| P_{g_h^*(t)} \sigma \|_{L^{s}(K)} \lesssim \| \sigma \|_{L^{s}(K)} , \quad t \in [0,T], \\  
    \label{eq:stab-Pgh-Pgh*}    
    & \| (P_{g_h^*(t)} - P_{g_h(t)}) \sigma \|_{L^s(K)} \lesssim 
\| g_h(t) - g_h^*(t) \|_{L^s(K)} \|\sigma\|_{L^{\infty}(K)}, \quad t \in [0,t^*].
\end{align}
\end{subequations}
} 
\end{lemma} 

\begin{lemma}[Approximation of $P_{g_h^*}$ and $\Pi_{g_h^*}$]\label{lm:app-Projection} 
    {Under the boundedness results in \eqref{eq:bound-g*}, let $\sigma \in \Sigma$ be a 
     piecewise smooth symmetric $(0,2)$-tensor field on $M_h$.  Then  it holds that  
    \begin{equation}\label{eq:app-Projection} 
    \| P_{g_h^*(t)} \sigma - \sigma \|_{L^\infty(M_h)} \lesssim h^{r+1}, \quad \text{and} \quad
    \| \Pi_{g_h^*(t)} \sigma - \sigma \|_{L^2(M_h)} \lesssim h^{r+1},
    \quad t \in [0,T].
    \end{equation} } 
\end{lemma}

\begin{remark}\label{rmk:L2-bound}
{\upshape 
For the $L^2$ stability estimate of both schemes \eqref{eq:num-Ric} and \eqref{eq:num-Ric-L2} with $r \geq 1$ and $q \geq 1$, we replace the bound in \eqref{eq:bound-e_g-e_k} by:
\begin{equation}\label{eq:bound-e_g-e_k-L2} 
    \| e_g \|_{L^{\infty}(0,t^*;L^2(M_h))} + 
    \| e_\kappa \|_{L^{\infty}(0,t^*;L^2(M_h))} \leq 
    h^{1.5}.
\end{equation} 
Then, by the same argument, the finite element solutions given by schemes \eqref{eq:num-Ric} and \eqref{eq:num-Ric-L2} with $r \geq 1$ and $q \geq 1$ satisfy
\begin{subequations}\label{eq:bound-gh-kh-L2}  
\begin{align}
\label{eq:bound-gh-L2} 
\hspace{-4mm}
\| g_h \|_{L^{\infty}(0,t^*;L^{\infty}(M_h))} + 
\| (g_h)^{-1} \|_{L^{\infty}(0,t^*;L^{\infty}(M_h))} + 
\| \sqrt{\det{g_h}} \|_{L^{\infty}(0,t^*;L^{\infty}(M_h))} \leq C,\\
\label{eq:bound-kh-L2} 
\| \kappa_h \|_{L^{\infty}(0,t^*;L^{2}(M_h))} \leq C. 
\end{align}
\end{subequations}
It is easy to see that the results of Lemma \ref{lm:norm-equi-gh}, Lemma \ref{lm:stab-Pgh-Pgh*}, and Lemma \ref{lm:app-Projection} remain valid.
}
\end{remark}

As a direct application of Lemma \ref{lm:stab-Pgh-Pgh*} and Lemma \ref{lm:app-Projection}, the following bounds for the {consistency defect}  functions hold.
\begin{lemma}[{Consistency defect estimates} ]\label{lm:defect} 
    Under the conditions of Theorem \ref{thm:main}, let $d_g, d_{L^2,g} \in \Sigma_h$ and $d_{\kappa} \in V_h$ denote the {consistency defect}  functions defined in  \eqref{eq:defect-g} , \eqref{eq:defect-g-L2} and \eqref{eq:defect-kappa}, respectively. The following estimates hold for $t \in [0,T]$,
\begin{subequations}
\begin{align}
    \label{eq:dg-esti-L2}  
    \|d_{L^2,g}(t)\|_{L^2(M_h)}  + \|d_g(t)\|_{L^2(M_h)}  & \lesssim h^{r+1} + h^{q+1}, \\ 
    \label{eq:dg-esti-Lp} 
    \|d_g(t)\|_{L^p(M_h)} &\lesssim (\ln{\frac{1}{h}}) h^{r+1} + (\ln{\frac{1}{h}})^{{\bar q}}h^{q+1}, \\ 
    \label{eq:esti-dk} 
    \|d_\kappa(t)\|_{L^2(M_h)}& \lesssim h^{r+1} + h^{q+1}.
\end{align}
\end{subequations}    
\end{lemma} 
\begin{proof}
    By pulling back the equation \eqref{eq:Ref-Ricci-a} to \( M_h \), we obtain
\[
\partial_t g^{(-l)}(t) + 2\big( \kappa^{(-l)}(t) - \bar \kappa \big) g^{(-l)}(t) = 0.
\]
Substituting this into \eqref{eq:defect-g}, we have:
\[ 
\begin{aligned}
    d_g = &(\partial_t g_h^* - \partial_t g^{(-l)}) +  
    2P_{g_h^*}( ( \kappa_h^* - \bar \kappa_h) (g_h^* -  g^{(-l)})) 
     \\ 
    & + 2P_{g_h^*}( ( \kappa_h^* -  \kappa^{(-l)})g^{(-l)})
    + 2P_{g_h^*}( (\bar \kappa - \bar \kappa_h )  g^{(-l)}) + 
    2(P_{g_h^*} - I)((\kappa^{(-l)} - \bar \kappa ) g^{(-l)}). 
\end{aligned}
\]
Take the \( L^p \) norm (for some \( p > 2 \)) and apply \eqref{eq:stab-inf} in Lemma \ref{lm:stab-Pgh-Pgh*} (Stability of \( P_{g_h} \) and \( P_{g_h^*} \)), Lemma \ref{lm:app-Projection} (Approximation of $P_{g_h^*}$ and $\Pi_{g_h^*}$), \eqref{eq:app-Ritz-Linf} in Lemma \ref{lm:app-Ritz} (Properties of the Ritz projection), \eqref{eq:prop-g_h^*}, the boundedness result in \eqref{eq:bound-g*-k*}, and the easily proven estimate $\| \bar \kappa_h -  \bar \kappa\|_{L^{\infty}(M_h)} \lesssim h^{q+1} + h^{r+1}$.  Then, the following holds for \( t \in [0,T] \):
\[ 
    \begin{aligned}
        \|d_g(t)\|_{L^p(M_h)} & \lesssim 
        \|\partial_t g_h^* - \partial_t g^{(-l)}\|_{L^p(M_h)} + 
        \| (\kappa_h^* - \bar \kappa_h) (g_h^* -  g^{(-l)})\|_{L^{\infty}(M_h)} \\ 
        & + 
        \|(\kappa_h^* -  \kappa^{(-l)}) g^{(-l)}\|_{L^{\infty}(M_h)} 
        + \| (\bar \kappa_h -  \bar \kappa )g^{(-l)}\|_{L^{\infty}(M_h)} + 
        h^{r+1}
        \\ & \lesssim (\ln{\frac{1}{h}}) h^{r+1} + (\ln{\frac{1}{h}})^{{\bar q}}h^{q+1}.
    \end{aligned}
\] 
To estimate the \( L^2 \) norm of \( d_g \), special attention needs to be paid to the third term of \( d_g \),
\[
\hspace{-3mm}
\begin{aligned}
&\quad \| P_{g_h^*}( ( \kappa_h^* -  \kappa^{(-l)})g^{(-l)})  \|_{L^2(M_h)} \\ 
& \lesssim 
\| ( \kappa_h^* -  I_h\kappa^{(-l)})g_h^*\|_{L^2(M_h)} + 
\| (I_h\kappa^{(-l)} - \kappa^{(-l)})g_h^* + ( \kappa_h^* -  \kappa^{(-l)})(g^{(-l)} - g_h^*)\|_{L^{\infty}(M_h)}\\ 
& \lesssim h^{q+1} + h^{r+1},
\end{aligned}
\]
where we have inserted the Lagrange interpolation \( I_h \kappa^{(-l)} \in V_h \) and utilized \eqref{eq:stab-Ls} to obtain the first term, and \eqref{eq:stab-inf} to obtain the second term in the above estimate. Thus, we have for $t \in [0,T]$, 
\[ 
    \begin{aligned}
        \|d_g(t)\|_{L^2(M_h)} & \lesssim 
        \|(\partial_t g_h^* - \partial_t g^{(-l)})\|_{L^2(M_h)} + 
        \|( \kappa_h^* - \bar \kappa_h)\|_{L^{\infty}(M_h)} \|(g_h^* -  g^{(-l)})\|_{L^{\infty}(M_h)} \\ 
        & + \| P_{g_h^*}( ( \kappa_h^* -  \kappa^{(-l)})g^{(-l)})  \|_{L^2(M_h)} +
        \|\bar \kappa_h -  \bar \kappa\|_{L^{\infty}(M_h)}\| g^{(-l)}\|_{L^{\infty}(M_h)} +
        h^{r+1}
        \\ & \lesssim h^{r+1} + h^{q+1}.
    \end{aligned}
\]
By a similar approach, it is straightforward to verify that the {consistency defect}  function \( d_{L^2, g} \) defined in \eqref{eq:defect-g-L2} also satisfies the above estimate.

For $\phi_h \in V_h$, 
we take the test function \( v = \phi_h^{(l)} \in H^1(\Man)\) in equation \eqref{eq:Ref-Ricci-b}, yielding
\[
\frac{\mathrm{d}}{\mathrm{d} t} (\kappa(t), \phi_h^{(l)})_{g(t), \mathcal{M}} 
+ (\nabla_{g(t)} \kappa(t), \nabla_{g(t)} \phi_h^{(l)})_{g(t), \mathcal{M}} = 0.
\]
Substituting this into \eqref{eq:defect-kappa}, we obtain:
\[
\begin{aligned}
(d_\kappa, \phi_h)_{g_h^*(t), M_h} &= \frac{\mathrm{d}}{\mathrm{d} t} (\kappa_h^*(t), \phi_h)_{g_h^*(t), M_h} 
- \frac{\mathrm{d}}{\mathrm{d} t} (\kappa(t), \phi_h^{(l)})_{g(t), \mathcal{M}} \\
&\quad + {(\nabla_{g_h^*}\kappa_h^*(t) , \nabla_{g_h^*(t)} \phi_h)_{g_h^*(t), M_h} 
- (\nabla_{g(t)} \kappa(t) , \nabla_{g(t)} \phi_h^{(l)})_{g(t), \mathcal{M}}}  \\
&= \frac{\mathrm{d}}{\mathrm{d} t} (\kappa_h^*(t), \phi_h)_{g_h^*(t), M_h} 
- \frac{\mathrm{d}}{\mathrm{d} t} (\kappa(t), \phi_h^{(l)})_{g(t), \mathcal{M}}.
\end{aligned}
\]
The last two terms cancel because \( \kappa_h^*(t) = R_h \kappa(t) \), and by the definition of the Ritz projection in \eqref{eq:Ritz-a}. Furthermore, we have: 
\[ \begin{aligned}
    \frac{\mathrm{d}}{\mathrm{d} t}(\kappa_h^*(t), \phi_h)_{g_h^*(t), M_h} &= 
     \int_{M_h} \big(
        (\partial_t\kappa_h^*) \phi_h \sqrt{\det g_h^*} +  
        \kappa_h^* \phi_h (\partial_t \sqrt{\det g_h^*})   
     \big)\mathrm{d}x,
    \\ 
    \frac{\mathrm{d}}{\mathrm{d} t}(\kappa(t), \phi_h^{(l)})_{g(t), \mathcal{M}} &=  \int_{M_h}  \big((\partial_t\kappa^{(-l)}) \phi_h \sqrt{\det g^{(-l)}} + 
    \kappa^{(-l)} \phi_h (\partial_t\sqrt{\det g^{(-l)}})\big)
    \mathrm{d}x.
    \\ 
\end{aligned} \]
Substituting the following identity for the time derivative of the determinant, 
\begin{equation}\label{eq:pt-det} 
    \partial_t \sqrt{\det g_h^*(t)} = \frac{1}{2} \sqrt{\det g_h^*(t)} (g_h^*(t))^{-1}: \partial_t g_h^*(t),
\end{equation}
and applying \eqref{eq:prop-g_h^*}, \eqref{eq:det-inv-app} and Lemma \ref{lm:app-Ritz} (Properties of  Ritz projection) as well as the boundedness result in \eqref{eq:bound-g*-k*}, we have 
\[ (d_\kappa, \phi_h)_{g_h^*(t), M_h} 
\lesssim (h^{r+1} + h^{q + 1})\| \phi_h\|_{L^2(g_h^*(t), M_h)}.  \]
Therefore, \eqref{eq:esti-dk} is proved by taking $\phi_h = d_\kappa$ and applying the norm equivalence in Lemma \ref{lm:norm-equ} (Norm equivalence of $g_h^*$).
\end{proof}

\subsection{Stability estimate}\label{sub-sec:stability}
This subsection establishes the stability estimates for the error equations in \eqref{eq:error-g}, \eqref{eq:error-g-L2}, and \eqref{eq:error-kappa}. 
We begin by considering the $L^p$ stability estimate of $e_g$ for scheme \eqref{eq:num-Ric} with $r \geq 0$ and $q \geq 1$. Recall the error equation \eqref{eq:error-g}, 
\[
\partial_t e_g = - 2P_{g_h}(e_{\kappa} g_h)  
-  2P_{g_h}((\kappa_h^* - \bar \kappa_h)e_g)  
- 2(P_{g_h} - P_{g_h^*})((\kappa_h^* - \bar \kappa_h) g_h^*(t)) - d_g.  
\]
Invoking \eqref{eq:stab-Pgh-Pgh*} in Lemma \ref{lm:stab-Pgh-Pgh*} (Stability of $P_{g_h}$ and $P_{g_h^*}$), we have for $s \in [2, +\infty]$ and $t \in [0,t^*]$
\begin{equation}\label{eq:esti-Pgh-Pgh*} 
\| (P_{g_h} - P_{g_h^*})( ( \kappa_h^* - \bar \kappa_h) g_h^*(t)) \|_{L^s(M_h)} 
\lesssim \| e_g \|_{L^s(M_h)} \|( \kappa_h^* - \bar \kappa_h) g_h^*(t) \|_{L^{\infty}(M_h)}.
\end{equation}  
Taking the $L^2$ norm in the above error equation, by \eqref{eq:esti-Pgh-Pgh*}, Lemma \ref{lm:stab-Pgh-Pgh*} (Stability of $P_{g_h}$ and $P_{g_h^*}$), as well as boundedness results in \eqref{eq:bound-g*-k*} and  \eqref{eq:bound-gh}, we have for $t \in [0, t^{*}]$,
\begin{equation}\label{eq:esti-p_te_g} 
    \begin{aligned}
        \| \partial_t e_g (t)\|_{L^2(M_h)}&  \lesssim 
    \|d_g\|_{L^2(M_h)} 
    + \|P_{g_h}(e_{\kappa} g_h)\|_{{L^2(M_h)}} \\ 
    & + \|P_{g_h}((\kappa_h^* - \bar \kappa_h)e_g)\|_{L^2(M_h)} 
    + \|(P_{g_h} - P_{g_h^*})( ( \kappa_h^* - \bar \kappa_h) g_h^*(t))\|_{L^2(M_h)} \\ 
    & \lesssim 
    \|d_g\|_{L^2(M_h)} + \|e_{\kappa}\|_{{L^2(M_h)}} + \|e_g\|_{L^2(M_h)}.  
    \end{aligned}
\end{equation} 
This, combined with the estimate in \eqref{eq:dg-esti-L2} and the bound in \eqref{eq:bound-e_g-e_k}, implies
\[ 
    \| \partial_t e_g\|_{L^{\infty}(0,t^*;L^2(M_h))} \leq C h^{\frac{1}{2} + \frac{1}{p}}.
\]
Furthermore, by a similar argument used in \eqref{eq:bound-kh}, for sufficiently small $h$, we have
\begin{equation}\label{eq:bound-p_tg_h} 
    \| \partial_t g_h \|_{L^{\infty}(0,t^*;L^{\tilde q}(M_h))} \leq C, \quad \text{ where }
    2 < \tilde q < \frac{4p}{p-2}.
\end{equation} 

\begin{proposition}[$L^p$ stability estimate of $e_g$ for scheme \eqref{eq:num-Ric}]\label{prop:stab-eg} 
    Under the conditions of Theorem \ref{thm:main}, for the scheme in \eqref{eq:num-Ric} with $r \geq 0, q \geq 1$, there
    exists a constant $h_0 > 0$ such that the following stability result holds for 
    $0<h \leq h_0$ and $t \in [0,t^*]$\rm{:} 
    \begin{equation}\label{eq:esti-eg} 
        \begin{aligned}
            \hspace{-2pt}
            \| e_g(t) \|_{L^p(M_h)}^2
    & \lesssim   \int_0^t \|d_g(s) \|^2_{L^p(M_h)} 
    +  \|e_{\kappa}(s) \|^2_{L^2(M_h)}
    + \|\nabla e_{\kappa}(s) \|^2_{L^2(M_h)}.
        \end{aligned}
    \end{equation} 
\end{proposition} 
\begin{proof}
Integrating \eqref{eq:error-g} from $0$ to $t$, we have 
\[  
    e_g(t) = 
    - 2 \int_0^t P_{g_h}(e_{\kappa} g_h) 
-  2 \int_0^tP_{g_h}((\kappa_h^* - \bar \kappa_h)e_g)
- 2 \int_0^t (P_{g_h} - P_{g_h^*})( ( \kappa_h^* - \bar \kappa_h) g_h^*(t)) - \int_0^t d_g.
\]
Taking the \( L^p \) norm and applying \eqref{eq:esti-Pgh-Pgh*}, Lemma \ref{lm:stab-Pgh-Pgh*} (stability of \( P_{g_h} \) and \( P_{g_h^*} \)), along with the boundedness results in \eqref{eq:bound-g*-k*} and \eqref{eq:bound-gh}, we obtain for \( t \in [0, t^{*}] \), 
\[ 
\begin{aligned}
    \| &e_g(t) \|_{L^p(M_h)}  \lesssim 
    \int_0^t \|d_g \|_{L^p(M_h)} + 
    \int_0^t \|g_h\|_{L^{\infty}(M_h)} \|e_{\kappa} \|_{L^p(M_h)} \\ 
    &+ 
    \int_0^t \|(\kappa_h^* - \bar \kappa_h)\|_{L^{\infty}(M_h)} \|e_g\|_{L^p(M_h)} + 
    \int_0^t \|( ( \kappa_h^* - \bar \kappa_h) g_h^*(t))\|_{L^{\infty}(M_h)} \|e_g\|_{L^p(M_h)}\\ 
    & \lesssim \int_0^t \|d_g \|_{L^p(M_h)} + 
    \int_0^t \|e_{\kappa} \|_{L^p(M_h)} + 
    \int_0^t \|e_g\|_{L^p(M_h)}.
\end{aligned}
\]
By Gr\"onwall's inequality, for  $t \in [0,t^*]$, we have 
\[     \begin{aligned}
    \| &e_g(t) \|_{L^p(M_h)}  \lesssim 
    \int_0^t \|d_g (s)\|_{L^p(M_h)} + 
\int_0^t \|e_{\kappa}(s) \|_{L^p(M_h)} \\ 
& \lesssim   (\int_0^t \|d_g(s) \|^2_{L^p(M_h)})^{1/2} 
+ (\int_0^t \|e_{\kappa}(s) \|^2_{L^2(M_h)})^{1/2} 
+ (\int_0^t \|\nabla e_{\kappa}(s) \|^2_{L^2(M_h)})^{1/2}, 
\end{aligned} \]
where in the last step, we used the Sobolev embedding $ H^1(M_h) \hookrightarrow L^p(M_h) $. 
\end{proof}
\begin{remark}\label{rmk:L2-stab-e_g}
    {\upshape 
    For the $L^2$ stability estimate of $e_g$ for both schemes \eqref{eq:num-Ric} and \eqref{eq:num-Ric-L2} with $r \geq 1, q \geq 1$, special attention needs to be paid to the error equation \eqref{eq:error-g-L2}, which involves the \( L^2 \) projection \( \Pi_{g_h(t)} \). In this case, the estimate in \eqref{eq:esti-Pgh-Pgh*} is replaced by the following estimate, 
    \begin{equation}\label{eq:esti-Pgh-Pgh*-L2} 
    \begin{aligned}
    &\quad \|(\Pi_{g_h} - \Pi_{g_h^*})((\kappa_h^* - \bar \kappa_h)g_h^*) \|_{L^2(M_h)} \\& \lesssim ( \|(\kappa_h^* - \bar \kappa_h)g_h^* \|_{L^{\infty}(M_h)} + 
    \|\partial_tg_h^*\|_{L^{\infty}(M_h)}) \|e_g\|_{L^2(M_h)} + \|d_{L^2,g}\|_{L^2(M_h)},
    \end{aligned}    
    \end{equation} 
    which can be proved as follows. For any $\rho \in \Sigma_h$, it holds that 
    \[ 
    \begin{aligned} 
    & \quad ((\Pi_{g_h} - \Pi_{g_h^*})((\kappa_h^*-\bar \kappa_h)g_h^*),\rho)_{g_h,M_h}\\ 
    &=((\kappa_h^*-\bar \kappa_h)g_h^*,\rho)_{g_h,M_h}
    - ((\kappa_h^*-\bar \kappa_h)g_h^*,\rho)_{g_h^*,M_h}\\ 
    & + (\Pi_{g_h^*}((\kappa_h^*-\bar \kappa_h)g_h^*),\rho)_{g_h^*,M_h}
    - (\Pi_{g_h^*}((\kappa_h^*-\bar \kappa_h)g_h^*),\rho)_{g_h,M_h}\\ 
    &\lesssim\|(\kappa_h^*-\bar \kappa_h)g_h^*\|_{L^{\infty}(M_h)}
      \|e_g\|_{L^2(M_h)} \| \rho \|_{L^2(M_h)} 
    \\
    & + {\big|}(\partial_tg_h^*-d_{L^2,g},\rho)_{g_h,M_h}
    - (\partial_tg_h^*-d_{L^2,g},\rho)_{g_h^*,M_h}{\big|}
    \\
    &\lesssim\Big((\|(\kappa_h^*-\bar \kappa_h)g_h^*\|_{L^{\infty}(M_h)}
    + \|\partial_tg_h^*\|_{L^{\infty}(M_h)})\|e_g\|_{L^2(M_h)}
    + \|d_{L^2,g}\|_{L^2(M_h)} \Big)\|\rho\|_{L^2(M_h)},
    \end{aligned} \]
    where we substitute equation \eqref{eq:error-g-L2} and apply \eqref{eq:gh-g*-inner-tensor} and \eqref{eq:gh-g*-equi-tensor} in Lemma \ref{lm:norm-equi-gh} (Norm equivalence of $g_h$). Then, \eqref{eq:esti-Pgh-Pgh*-L2} is proved by setting $\rho = (\Pi_{g_h} - \Pi_{g_h^*})((\kappa_h^* - \bar{\kappa}_h) g_h^*)$. 
    Now, applying a similar approach as in \eqref{eq:esti-p_te_g} while utilizing \eqref{eq:bound-gh-kh-L2} and \eqref{eq:bound-e_g-e_k-L2} from Remark \ref{rmk:L2-bound}, we obtain 
\[
\|\partial_t e_g\|_{L^{\infty}(0, t^*; L^{2}(M_h))} \lesssim h^{1.5},
\]
which consequently implies  
\begin{equation}\label{eq:bound-p_tg_h-L2} 
    \|\partial_t g_h\|_{L^{\infty}(0, t^*; L^{\infty}(M_h))} \lesssim 
    \|\partial_t g_h^*\|_{L^{\infty}(0, t^*; L^{\infty}(M_h))} + 
    h^{-1} \|\partial_t e_g\|_{L^{\infty}(0, t^*; L^{2}(M_h))} \leq 
    C,
\end{equation}
for the finite element solutions given by schemes \eqref{eq:num-Ric} and \eqref{eq:num-Ric-L2} with \(r \geq 1, q \geq 1\).
    Moreover, employing a similar approach as in Proposition \ref{prop:stab-eg} with \eqref{eq:esti-Pgh-Pgh*-L2}, we derive the following \(L^2\) stability estimates for \(e_g\) on \([0,t^*]\),
    \begin{subequations}\label{eq:L2-esti-eg}
        \begin{align}
        \label{eq:L2-esti-eg-Regge} 
            \| e_g(t) \|_{L^2(M_h)}^2 &\lesssim 
            \int_0^t \|d_g(s) \|^2_{L^2(M_h)}
            +  \|e_{\kappa}(s) \|^2_{L^2(M_h)},\\
        \label{eq:L2-esti-eg-L2} 
        \| e_g(t) \|_{L^2(M_h)}^2 &\lesssim 
        \int_0^t \|d_{L^2, g}(s) \|^2_{L^2(M_h)}
        +  \|e_{\kappa}(s) \|^2_{L^2(M_h)},   
        \end{align}
    \end{subequations}
    corresponding to the finite element solutions given by \eqref{eq:num-Ric} and \eqref{eq:num-Ric-L2} respectively.
}
\end{remark}

We now turn to the stability estimate of \( e_{\kappa} \) in the error equation \eqref{eq:error-kappa}. By utilizing the matrix-vector representation, the following holds:
{\[ \begin{aligned}
    \frac{1}{2}\frac{\mathrm{d}}{\mathrm{d}t} ( \| e_\kappa \|^2_{L^2(g_h, M_h)}) &+ 
    \| \nabla_{g_h} e_\kappa \|_{L^2(g_h, M_h)}^2 = 
    \frac{1}{2}\frac{\mathrm{d}}{\mathrm{d}t} \big( \ekappa^{\top} \M[\g(t)] \ekappa \big) + 
    \ekappa^{\top} \A[\g(t)] \ekappa \\ 
    &= \ekappa^{\top} \big(\frac{\mathrm{d}}{\mathrm{d}t}\big(\M[\g(t)] \ekappa \big) +  \A[\g(t)] \ekappa  \big)- \frac{1}{2} \ekappa^{\top} \dot{(\M[\g(t)])} \ekappa.
\end{aligned} \]} 
Then, substituting the error equation \eqref{eq:error-kappa}, we have,
\begin{equation}\label{eq:esti-ekappa} 
\begin{aligned}
    \frac{1}{2}\frac{\mathrm{d}}{\mathrm{d}t} ( \| e_\kappa \|^2_{L^2(g_h, M_h)}) &+ 
    \| \nabla_{g_h} e_\kappa \|_{L^2(g_h, M_h)}^2 = \ekappa^{\top}{  ( \dot{(\M[\gs(t)])} - \dot{(\M[\g(t)])} )} \ka^*(t) \\ 
    & + \ekappa^{\top} { ( {\M[\gs(t)]} - {\M[\g(t)]} )} \dot{\ka^*(t)}
    + \ekappa^{\top} { ( {\A[\gs(t)]} - {\A[\g(t)]} )} \ka^*(t) \\
    & - \frac{1}{2} \ekappa^{\top} { \dot{(\M[\g(t)])}}  \ekappa - \ekappa^{\top} \M[\gs(t)] \dkappa  \\
    & = I_1 + I_2 + I_3 + I_4  - \ekappa^{\top} \M[\gs(t)] \dkappa.
\end{aligned}
\end{equation} 
To handle \( I_1, \dots, I_4 \), we introduce the intermediate metric between \( g_h(t) \) and \( g_h^*(t) \), defined as
\[
g_{h,\theta}(t) = \theta g_h(t) + (1 - \theta) g_h^*(t) \in \Sigma_h, \quad \theta \in [0, 1].
\]
The following identities thus hold:
\begin{subequations}\label{eq:g-theta}  
\begin{align}
    \label{eq:inv-g} 
    \hspace{-10mm}
    (g_h(t))^{-1} - (g_h^*(t))^{-1} 
    & = -\int_0^1  g_{h, \theta}^{-1}(t) \frac{\mathrm{d}}{\mathrm{d} \theta} ( g_{h, \theta}(t) )g_{h, \theta}^{-1}(t) = -\int_0^1  g_{h, \theta}^{-1}(t) { e_g(t)}  g_{h, \theta}^{-1}(t) , \\ 
    \label{eq:det-g} 
    \hspace{-8mm}
    \sqrt{\det{g_h(t)}} &- \sqrt{\det{g^*_h(t)}} = \frac{1}{2} \int_0^1 \sqrt{\det g_{h,\theta}(t)} 
    g_{h, \theta}^{-1}(t):{ e_g(t)}. 
\end{align}
\end{subequations}
From \eqref{eq:inv-det} and the boundedness results in \eqref{eq:bound-g*} and \eqref{eq:bound-gh}, we obtain the following bound for the intermediate metric \( g_{h,\theta} \) with \( \theta \in [0, 1] \):
\begin{equation}\label{eq:bound-gh^thete} 
    \| g_{h, \theta} \|_{L^{\infty}(0,t^*;L^{\infty}(M_h))} + 
    \| g_{h, \theta}^{-1} \|_{L^{\infty}(0,t^*;L^{\infty}(M_h))} + 
    \| \sqrt{\det{g_{h, \theta}}} \|_{L^{\infty}(0,t^*;L^{\infty}(M_h))} \leq C.
\end{equation} 
Based on \eqref{eq:inv-g}, \eqref{eq:det-g}, the terms \( I_1, \dots, I_4 \) in \eqref{eq:esti-ekappa} can be expressed using the following lemma.
\begin{lemma}\label{lm:g_thete} 
For any $v_h, w_h \in V_h$ with vector representations $\bv, \w \in \mathbb{R}^{{\rm dim}V_h}$, the following equalities hold:
\begin{subequations} 
    \begin{align}   
        \label{eq:M-M*} 
        \bv^\top ( {\M[\gs(t)]} - {\M[\g(t)]} ) \w 
        &= -\frac{1}{2} \int_{M_h} v_h w_h\big( 
            \int_0^1 \sqrt{\det g_{h,\theta}} 
            g_{h, \theta}^{-1}\big):{ e_g}\mathrm{d}x, 
    \end{align}
    \begin{align}  
    \label{eq:dot-M} 
    \bv^\top \dot{(\M[\g(t)])} \w = \frac{1}{2}\int_{M_h} v_h w_h \big( \sqrt{\det g_h} (g_h)^{-1}: (\partial_t g_h)  \big)\mathrm{d}x,
    \end{align}
    \begin{align}\label{eq:dot-M-M*} 
        \bv^\top (\dot{(\M[\gs(t)])} - \dot{(\M[\g(t)])}) \w
    & = -\frac{1}{4}\int_{M_h} v_h w_h 
    \big( \int_0^1 \sqrt{\det g_{h,\theta}  } g_{h, \theta}^{-1}  :{ e_g  } \big) 
    (g^*_h  )^{-1}: (\partial_t g^*_h)\mathrm{d}x  \notag \\
    & - \frac{1}{2}\int_{M_h} v_h w_h \sqrt{\det g_h  }
    (g^*_h  )^{-1}: ({ \partial_t e_g}) \mathrm{d}x \notag \\ 
    & + \frac{1}{2}\int_{M_h} v_h w_h \sqrt{\det g_h  }
    \big( \int_0^1  g_{h, \theta}^{-1}   { e_g  }  g_{h, \theta}^{-1}   \big) : (\partial_t g_h) \mathrm{d}x,
    \end{align}
    \begin{align}\label{eq:A-A*} 
        \bv^\top (\A[\gs (t)] - \A[\g (t)]) \w &= 
        -\frac{1}{2} \int_{M_h} (g^*_h)^{-1}\nabla v_h \cdot \nabla w_h 
        \big(  \int_0^1 \sqrt{\det g_{h,\theta}} g_{h, \theta}^{-1}\big):{ e_g}\mathrm{d}x \notag \\
        & + \int_{M_h} \big( \int_0^1  g_{h, \theta}^{-1} { e_g}  g_{h, \theta}^{-1} \big)\nabla v_h \cdot \nabla w_h \sqrt{\det g_h} \mathrm{d}x.
    \end{align}
\end{subequations}
\end{lemma} 
\begin{proof}
    \eqref{eq:M-M*} follows from the definition of the metric-dependent mass matrices $\M[\gs(t)]$, $\M[\g(t)]$, and the identity \eqref{eq:det-g}.
    \eqref{eq:dot-M} follows from the identity \eqref{eq:pt-det}.
    For \eqref{eq:dot-M-M*}, it holds that
\[ \begin{aligned}
    \bv^\top (\dot{(\M[\gs(t)])} - \dot{(\M[\g(t)])}) \w 
& = \frac{1}{2}\int_{M_h} v_h w_h \big( \sqrt{\det g^*_h } - \sqrt{\det g_h }\big)
(g^*_h )^{-1}: (\partial_t g^*_h)\mathrm{d}x \\
& + \frac{1}{2}\int_{M_h} v_h w_h \sqrt{\det g_h }
(g^*_h )^{-1}: \big((\partial_t g^*_h) - (\partial_t g_h)\big)\mathrm{d}x \\
& + \frac{1}{2}\int_{M_h} v_h w_h \sqrt{\det g_h }
\big((g^*_h )^{-1} - (g_h )^{-1}\big) : (\partial_t g_h)\mathrm{d}x .
\end{aligned} \]
Therefore, \eqref{eq:dot-M-M*} is proved by using \eqref{eq:inv-g} and \eqref{eq:det-g}.
    For \eqref{eq:A-A*}, it holds that
    \[ \begin{aligned}
    \bv^\top \A[\gs (t)] \w - \bv^\top \A[\g (t)] \w 
   &=  \int_{M_h} (g^*_h )^{-1}\nabla v_h \cdot \nabla w_h \big(\sqrt{\det g^*_h } - \sqrt{\det g_h }\big)\mathrm{d}x\\ 
   &+ \int_{M_h} \big((g^*_h )^{-1} - (g_h )^{-1}\big)\nabla v_h \cdot \nabla w_h \sqrt{\det g_h }\mathrm{d}x.
\end{aligned}\]
Therefore, \eqref{eq:A-A*} is proved by using \eqref{eq:inv-g} and \eqref{eq:det-g}.
\end{proof}

\begin{proposition}[Stability estimate of $e_\kappa$ for scheme \eqref{eq:num-Ric}]\label{prop:stab-ekappa} 
    Under the conditions of Theorem \ref{thm:main}, for the scheme in \eqref{eq:num-Ric} with $r \geq 0, q \geq 1$, there
    exists a constant $h_0 > 0$ such that the following stability result holds for 
    $0<h \leq h_0$ and $t \in [0,t^*]$\rm{:}
    \begin{subequations}
        \begin{align}
            \label{eq:esti-ekappa-LinfL2} 
            \| e_\kappa(t) \|^2_{L^2 (M_h)} &\lesssim 
            \int_0^t \|e_g(s)\|_{L^p(M_h)}^2 + 
            \|d_g(s)\|_{L^2(M_h)}^2 + \|d_\kappa(s)\|_{L^2(M_h)}^2,\\
            \label{eq:esti-ekappa-L2H1} 
         \int_0^t \| \nabla e_\kappa(s) \|^2_{L^2 (M_h)} 
         &\lesssim \int_0^t \|e_g(s)\|_{L^p(M_h)}^2 + \|e_\kappa(s)\|_{L^2(M_h)}^2 \\  
        & + \int_0^t\|d_g(s)\|_{L^2(M_h)}^2 + \|d_\kappa(s)\|_{L^2(M_h)}^2.  \notag
        \end{align}
        \end{subequations}    
\end{proposition} 
\begin{proof}
We apply Lemma \ref{lm:g_thete} to estimate $I_1, \dots, I_4$ in \eqref{eq:esti-ekappa}. 
For $I_1$, using \eqref{eq:dot-M-M*}, we have 
{\[ 
\begin{aligned}
    \hspace{-100pt}
    I_1 & = \ekappa^{\top}{( \dot{(\M[\gs(t)])} - \dot{(\M[\g(t)])} )} \ka^*(t) \\
    & \lesssim 
    \|{ e_\kappa} \|_{L^2(M_h)}\big(  
        (\int_0^1\|\sqrt{\det g_{h,\theta}} g_{h, \theta}^{-1}\|_{L^\infty(M_h)}) 
        \|\kappa^*_h(g^*_h  )^{-1} : \partial_t g^*_h \|_{L^\infty(M_h)}
    \big) \|{ e_g  }\|_{L^2(M_h)} \\ 
    &+ \|{ e_\kappa} \|_{L^2(M_h)}\big( 
        \|\kappa^*_h \sqrt{\det g_h  }(g^*_h  )^{-1}\|_{L^\infty(M_h)}
    \big) \|{ \partial_t e_g}\|_{L^2(M_h)} \\
    &+ \|{ e_\kappa} \|_{L^2(M_h)}\big( 
        \|\kappa^*_h \sqrt{\det g_h  }\|_{L^\infty(M_h)}  (\int_0^1 \|g_{h, \theta}^{-1}\|_{L^\infty(M_h)}^2)  
    \big) \|\partial_t g_h \|_{{ L^{p^{\prime}}(M_h)} } \|{ e_g  }\|_{{ L^p(M_h)} }. 
\end{aligned}
\]} 
Here $p^{\prime} = \frac{2p}{p-2} $, thus we have  
$2 < p^{\prime} < \frac{4p}{p-2}$ .
According to the boundedness results in \eqref{eq:bound-g*-k*}, \eqref{eq:bound-gh}, \eqref{eq:bound-p_tg_h} and \eqref{eq:bound-gh^thete}, it holds that for $t \in [0,t^{*}]$,
\[ I_1 \lesssim 
\|e_\kappa(t)\|_{L^2(M_h)}^2 + \|e_g(t)\|_{L^p(M_h)}^2 + \|\partial_t e_g(t)\|_{L^2(M_h)}^2.\]
We use \eqref{eq:M-M*}  to estimate $I_2=\ekappa^{\top} {( {\M[\gs(t)]} - {\M[\g(t)]} )} \dot{\ka^*(t)} $:
{\[ \begin{aligned}
    I_2 
    & \lesssim  \| { e_\kappa}\|_{L^2(M_h)}
    \big( \|\partial_t\kappa_h^*\|_{L^\infty(M_h)}(\int_0^1 \|\sqrt{\det g_{h,\theta} } \|_{L^\infty(M_h)}
    \|g_{h, \theta}^{-1} \|_{L^\infty(M_h)})
    \big)  \|{ e_g }\|_{L^2(M_h)}.
\end{aligned} \]} 
According to the boundedness results in \eqref{eq:bound-g*-k*} and \eqref{eq:bound-gh^thete}, it holds that for $t \in [0,t^{*}]$, 
\[ I_2 \lesssim 
\|e_\kappa(t)\|_{L^2(M_h)}^2 + \|e_g(t)\|_{L^2(M_h)}^2.\]
For $I_3 = \ekappa^{\top} {( {\A[\gs(t)]} - {\A[\g(t)]} )} \ka^*(t)$, using \eqref{eq:A-A*}, we have
{\[ \begin{aligned}
    \hspace{-120pt}
   I_3 & \lesssim 
    \|{ \nabla e_\kappa}\|_{L^2(M_h)}\big(
    \|(g^*_h )^{-1} \nabla  \kappa^*_h\|_{L^\infty(M_h)}
    (\int_0^1\|(\sqrt{\det g_{h,\theta} }) g_{h, \theta}^{-1} \|_{L^\infty(M_h)})
    \big) \|{ e_g }\|_{L^2(M_h)}\\
   & +   \|{ \nabla e_\kappa}\|_{L^2(M_h)}\big(
(\int_0^1\|g_{h, \theta}^{-1}  \|_{L^\infty(M_h)}^2)
\|\nabla  \kappa^*_h \|_{L^\infty(M_h)}
\|\sqrt{\det g_h }\|_{L^\infty(M_h)}
   \big) \|{ e_g }\|_{L^2(M_h)}.
\end{aligned} \]} 
According to the boundedness results in \eqref{eq:bound-g*-k*}, \eqref{eq:bound-gh} and \eqref{eq:bound-gh^thete}, it holds that for $t \in [0,t^{*}]$, 
\[ I_3 \lesssim 
\varepsilon \|\nabla e_\kappa(t)\|_{L^2(M_h)}^2 + \varepsilon^{-1} \|e_g(t)\|_{L^2(M_h)}^2,\]
where $\varepsilon > 0$ can be arbitrary small.
We apply \eqref{eq:dot-M}  to estimate $I_4$,
\[ \begin{aligned}
    I_4& = - \frac{1}{2} \ekappa^{\top} { \dot{(\M[\g(t)])}}  \ekappa 
    = -\frac{1}{4}\int_{M_h} e_\kappa e_\kappa  \big( \sqrt{\det g_h(t)} (g_h(t))^{-1}: (\partial_t g_h)  \big)\mathrm{d}x \\ 
    &\lesssim  \big(  \|\sqrt{\det g_h}\|_{L^\infty(M_h)} \|(g_h)^{-1}\|_{L^\infty(M_h)} \big)
    \|\partial_t g_h \|_{{ L^{p^{\prime}}(M_h)}}
    \|{ e_\kappa} \|_{{ L^p(M_h)} }
    \|{ e_\kappa} \|_{L^2(M_h)}.
\end{aligned} \]
Similarly to $I_1$, 
according to \eqref{eq:bound-p_tg_h} and  \eqref{eq:bound-gh}, it holds that for $t \in [0,t^{*}]$,
\[ I_4 \lesssim  \|{ e_\kappa} \|_{{ L^p(M_h)} }
\|{ e_\kappa} \|_{L^2(M_h)} \lesssim  \|{ e_\kappa} \|_{{ H^1(M_h)} }
\|{ e_\kappa} \|_{L^2(M_h)}\lesssim \varepsilon \|\nabla{e_\kappa} \|_{{L^2(M_h)} }^2 + \varepsilon^{-1} \|{e_\kappa} \|_{L^2(M_h)}^2, \]
where we have used the Sobolev embedding $ H^1(M_h) \hookrightarrow L^p(M_h) $.

Combining the above estimates of $I_1, \cdots, I_4$ with \eqref{eq:esti-ekappa}, we have, for $t \in [0,t^*]$,
\[ 
\begin{aligned}
    & \quad \frac{1}{2}\frac{\mathrm{d}}{\mathrm{d}t} ( \| e_\kappa \|^2_{L^2(g_h, M_h)}) + 
\| \nabla_{g_h} e_\kappa \|_{L^2(g_h, M_h)}^2 \\
& \lesssim 
\|e_g(t)\|_{L^p(M_h)}^2 + \|e_\kappa(t)\|_{L^2(M_h)}^2 + \|d_\kappa(t)\|_{L^2(M_h)}^2 
+ \| \partial_t e_g(t) \|_{L^2}^2 + \varepsilon  \|\nabla e_\kappa(t)\|_{L^2(M_h)}^2.\\ 
\end{aligned}
\]
On the right-hand side, the term \( \| \partial_t e_g(t) \|_{L^2(M_h)}^2 \) can be further estimated using \eqref{eq:esti-p_te_g}, while the term \( \varepsilon  \|\nabla e_\kappa(t)\|_{L^2(M_h)}^2 \) can be absorbed into the left-hand side. Consequently, we obtain
\[ 
\begin{aligned}
    & \quad \frac{1}{2}\frac{\mathrm{d}}{\mathrm{d}t} ( \| e_\kappa \|^2_{L^2(g_h, M_h)}) + 
\| \nabla_{g_h} e_\kappa \|_{L^2(g_h, M_h)}^2 \\
& \lesssim
\|e_g(t)\|_{L^p(M_h)}^2 + \|e_\kappa(t)\|_{L^2(M_h)}^2 + 
\|d_g(t)\|_{L^2(M_h)}^2 + \|d_\kappa(t)\|_{L^2(M_h)}^2.
\end{aligned}
\]
{By applying Gr\"onwall's inequality and the norm equivalence in Lemma \ref{lm:norm-equi-gh}, and then integrating over \([0,t]\) for \(t\in[0,t^*]\), we obtain the estimates in \eqref{eq:esti-ekappa-LinfL2} and \eqref{eq:esti-ekappa-L2H1}. } 
\end{proof}
\begin{remark}\label{rmk:L2-stab-e_k}
    {\upshape 
    To establish the \( L^2 \) stability estimate for both schemes \eqref{eq:num-Ric} and \eqref{eq:num-Ric-L2} with \( r \geq 1, q \geq 1 \), we employ a similar approach as in Proposition \ref{prop:stab-ekappa}, utilizing \eqref{eq:bound-p_tg_h-L2} from Remark \ref{rmk:L2-stab-e_g}, as well as \eqref{eq:bound-gh-kh-L2} and \eqref{eq:bound-e_g-e_k-L2} from Remark \ref{rmk:L2-bound}. This leads to the following \( L^2 \) stability estimates for \( e_\kappa \) on \([0,t^*]\):
    \begin{subequations}
    \begin{align}
        \label{eq:L2-esti-ek-Regge} 
        \|e_\kappa (t) \|^2_{L^2(M_h)} \lesssim 
        \int_0^t \|d_g(s) \|^2_{L^2(M_h)}
    + \int_0^t \|d_\kappa(s) \|^2_{L^2(M_h)}
    + \int_0^t \|e_{g}(s) \|^2_{L^2(M_h)}, \\
    \label{eq:L2-esti-ek-L2} 
    \|e_\kappa (t) \|^2_{L^2(M_h)} \lesssim 
    \int_0^t \|d_{L^2,g}(s) \|^2_{L^2(M_h)}
+ \int_0^t \|d_\kappa(s) \|^2_{L^2(M_h)}
+ \int_0^t \|e_{g}(s) \|^2_{L^2(M_h)}, 
    \end{align}
    \end{subequations}
    corresponding to the finite element solutions given by \eqref{eq:num-Ric} and \eqref{eq:num-Ric-L2} respectively.}
\end{remark}

\section{Error estimate}\label{sec:err}

In this section, we first prove Theorem \ref{thm:main} by combining the stability estimates from Proposition \ref{prop:stab-eg} and Proposition \ref{prop:stab-ekappa} with the {consistency defect estimates}  in Lemma \ref{lm:defect}. We then prove Theorem \ref{thm:main-L2} in a similar manner, using the stability results from Remark \ref{rmk:L2-stab-e_g} and Remark \ref{rmk:L2-stab-e_k}.

\noindent \textit{Proof of Theorem \ref{thm:main}:}    
Combining \eqref{eq:esti-eg}, \eqref{eq:esti-ekappa-LinfL2} and \eqref{eq:esti-ekappa-L2H1} we have , for $t\in [0,t^*]$
\[ \begin{aligned}
     \| e_g(t) \|^2_{L^p(M_h)}  +  \| e_\kappa(t) \|^2_{L^2 (M_h)} 
& \lesssim \int_0^t \big(\|e_g(s)\|_{L^p(M_h)}^2 + \|e_\kappa(s)\|_{L^2(M_h)}^2\big) \\ 
& + \int_0^t \big(\|d_g(s) \|^2_{L^p(M_h)}  +  \|d_g(s)\|_{L^2(M_h)}^2 + \|d_\kappa(s)\|_{L^2(M_h)}^2\big).
\end{aligned} \]
By Gr\"onwall's inequality, for $t\in [0,t^*]$, we have 
\[ \begin{aligned}
    & \quad \| e_g(t) \|^2_{L^p(M_h)}  +  \| e_\kappa(t) \|^2_{L^2 (M_h)} \\
& \lesssim \int_0^t \big(\|d_g(s) \|^2_{L^p(M_h)}  +  \|d_g(s)\|_{L^2(M_h)}^2 + \|d_\kappa(s)\|_{L^2(M_h)}^2\big).
\end{aligned} \]
Then, the estimates \eqref{eq:dg-esti-L2}, \eqref{eq:dg-esti-Lp}, and \eqref{eq:esti-dk} imply that for $t \in [0,t^*]$  
\begin{equation}\label{eq:error-estimate}  
    \| e_g(t) \|_{L^p(M_h)}  +  \| e_\kappa(t) \|_{L^2 (M_h)}  \lesssim (\ln(\frac{1}{h}))^{\bar{q}}h^{q+1} + \ln(\frac{1}{h}) h^{r+1}.  
\end{equation}  
Since {$r \geq 0$ and $q \geq 1$}, by the continuity of the spatially semidiscrete finite element solution in time (which is essentially the solution of an ODE problem), the above estimate implies that \eqref{eq:bound-e_g-e_k} still holds in a bigger interval $[0,t^* + \delta]$ for some $ \delta > 0$. Since $t^* \in (0, T ]$ is the supremum of times such that \eqref{eq:bound-e_g-e_k} holds for $t \in [0,t^*]$, it follows that $t^* = T$ and \eqref{eq:error-estimate} holds for all $t \in [0,T]$. 

The proof of Theorem \ref{thm:main} is concluded by combining \eqref{eq:error-estimate}  with the approximation properties in \eqref{eq:prop-g_h^*} and Lemma \ref{lm:app-Ritz} (properties of the Ritz projection).


\noindent \textit{Proof of Theorem \ref{thm:main-L2}:}     
We provide the proof only for the scheme \eqref{eq:num-Ric}, as the proof for scheme \eqref{eq:num-Ric-L2} is the same and thus omitted.
In the case of $r \geq 1, q \geq 1$, combining \eqref{eq:L2-esti-eg-Regge} in Remark \ref{rmk:L2-stab-e_g} and \eqref{eq:L2-esti-ek-Regge} in Remark \ref{rmk:L2-stab-e_k}, we have for $t \in [0,t^*]$,
\[ 
\begin{aligned}
    & \quad \| e_g(t) \|_{L^2(M_h)}^2 +  \|e_\kappa (t) \|_{L^2(M_h)}^2 \\
    & \lesssim \int_0^t (  \|d_g(s) \|^2_{L^2(M_h)} + \|d_\kappa(s) \|^2_{L^2(M_h)}) 
    + \int_0^t( \| e_g(s) \|_{L^2(M_h)}^2 +  \|e_\kappa (s) \|_{L^2(M_h)}^2). 
\end{aligned}
\]
The estimates  \eqref{eq:dg-esti-L2}, \eqref{eq:esti-dk} and the Gr\"onwall's inequality imply that for $t \in [0, t^*]$
\begin{equation}\label{eq:error-estimate-L2} 
    \| e_g(t) \|_{L^2(M_h)} + \| e_\kappa(t) \|_{L^2(M_h)} \lesssim h^{q+1} + h^{r+1}.
\end{equation} 
Since \( q \geq 1 \) and \( r \geq 1 \), by applying the same argument as in Theorem \ref{thm:main}, we conclude that \( t^* = T \) in \eqref{eq:bound-e_g-e_k-L2} of Remark \ref{rmk:L2-bound}, and that \eqref{eq:error-estimate-L2} holds for all \( t \in [0,T] \).
Therefore, the proof of Theorem \ref{thm:main-L2} is concluded by combining \eqref{eq:error-estimate-L2} with the approximation properties in \eqref{eq:prop-g_h^*} and Lemma \ref{lm:app-Ritz} (Properties of  Ritz projection).


\section{{Computing an embedding} } \label{sec:embedding}

{This section introduces an approach for determining an evolving surface $\mathcal{M}(t) \subset \mathbb{R}^3$ that realizes the given evolving metric $g(t)$.  
We begin by deriving the PDE governing the surface velocity, which we then discretize using the parametric FEM.  
} 
Given an evolving metric $g(t)$ on $\mathcal{M}(0) := \mathcal{M}$, we want to construct a family of surfaces $\mathcal{M}(t) \subset \mathbb{R}^3$ and a map $\varphi_t : \mathcal{M}(0) \to \mathcal{M}(t)$ so that 
\begin{equation} \label{eq:embed}
\varphi_t^* i^* \delta_{\mathbb{R}^3} = g(t),
\end{equation}  
where $i^*$ represents the pullback under the inclusion $i : \mathcal{M}(t) \hookrightarrow \mathbb{R}^3$, and $\varphi_t^*$ represents the pullback under $\varphi_t$. 
{Let \(n = n(\cdot, t): \Man(t) \rightarrow \mathbb{R}^3\) denote the unit normal vector to \(\Man(t)\), and let \({\rm P} = I - n n^{\top}: \Man(t) \rightarrow \mathbb{R}^{3 \times 3}\) denote the tangential projection matrix. The surface gradient on \(\Man(t)\) is denoted by \(\nabla_{\Man(t)}\). }
Differentiating~\eqref{eq:embed} with respect to $t$ and applying the pushforward $\varphi_{t*} = (\varphi_t^{-1})^*$ yields  
\begin{equation} \label{eq:embed_dot}
    \mathcal{L}_V i^* \delta_{\mathbb{R}^3} = \varphi_{t*} \dot{g}, 
\end{equation}  
where $\mathcal{L}_V$ denotes the Lie derivative with respect to the velocity field  
$V = \dot{\varphi}_t \circ \varphi_t^{-1}$.  
Using the identity \(\mathcal{L}_V i^* \delta_{\mathbb{R}^3} = 2 \nabla_{\Man(t)}^{\mathrm{sym}} V\), 
where \(\nabla_{\Man(t)}^{\mathrm{sym}}\) denotes the symmetrized covariant derivative, {which is defined by
\(\nabla_{\Man(t)}^{\mathrm{sym}} V = \frac{1}{2} {\rm P} \left( \nabla_{\Man(t)} V + (\nabla_{\Man(t)} V)^{\top} \right) {\rm P}\)}, we obtain
\begin{equation} \label{eq:embed_dot1}
    \nabla_{\Man(t)}^{\mathrm{sym}} V = \frac{1}{2} \varphi_{t*} \dot{g}.
\end{equation}
Now let us decompose $V$ into its tangential and normal components by writing
\[
V = V^{||} + un,
\]
where $V^{||} = {\rm P}V$ is tangent to $\mathcal{M}(t)$ and $u = V \cdot n$. Substituting this into \eqref{eq:embed_dot1}, we obtain the PDE that the surface velocity should satisfy:
\begin{equation} \label{eq:embed_dot2}
\nabla^{\rm sym}_{\mathcal{M}(t)} V^{||} + u \sff = \frac{1}{2} \varphi_{t*} \dot{g},
\end{equation}
where $\sff = \nabla_{\Man(t)} n $ is the second fundamental form. 
However, the PDE \eqref{eq:embed_dot2} is underdetermined; therefore, we propose the following problem to uniquely determine a solution. 
We introduce a symmetric positive definite bilinear  form \( a_{\mathcal{M}(t)}(\cdot, \cdot) \) on \(H^1(\mathcal{M}(t))^3\) and we then consider the following problem:
\[ \begin{aligned}
    & \min \frac{1}{2} a_{\mathcal{M}(t)}(V, V),  \\
    \text{s.t. } & \nabla^{\rm sym}_{\mathcal{M}(t)} V^{||} + u \sff = \frac{1}{2} \varphi_{t*} \dot{g}.
\end{aligned} \]
This can be reformulated as the following PDE for the surface velocity:  
Find a velocity field $V \in H^1(\mathcal{M}(t))^3$ and a symmetric $(0,2)$-tensor field $\sigma$ on $\Man(t)$ such that   
\begin{subequations}\label{eq:embed-saddle-PDE}  
    \begin{align}
        \label{eq:embed-saddle-PDE-a} 
        a_{\mathcal{M}(t)}(V, W) + (\sigma, \nabla^{\rm sym}_{\mathcal{M}(t)} W^{||} + w \sff)_{\mathcal{M}(t)} &= 0, \\ 
        \label{eq:embed-saddle-PDE-b} 
        (\nabla^{\rm sym}_{\mathcal{M}(t)} V^{||} + u \sff, \rho)_{\mathcal{M}(t)} &= 
        \frac{1}{2}(\varphi_{t*} \dot{g}, \rho)_{\mathcal{M}(t)},
    \end{align}
\end{subequations}  
holds for all  $W \in H^1(\mathcal{M}(t))^3$ and all symmetric $(0,2)$-tensor fields $\rho$ on $\Man(t)$,  
where \( W = W^{||} + w n \), as in the previous decomposition.

\begin{remark}(Choice of $a_{\mathcal{M}(t)}(\cdot,\cdot)$)
    \upshape
    The choice of \( a_{\mathcal{M}(t)}(\cdot,\cdot) \) is crucial for ensuring the stability of the PDE \eqref{eq:embed-saddle-PDE}. Here, we define  
    \[
    a_{\mathcal{M}(t)}(V,W) = 
    (\nabla^{\rm sym}_{\mathcal{M}(t)} V^{||}, \nabla^{\rm sym}_{\mathcal{M}(t)} W^{||})_{L^2(\mathcal{M}(t))} + 
    (u,w)_{L^2(\mathcal{M}(t))},
    \]
    where \( V =  V^{||} + u n \) and \( W = W^{||} + w n \),  
    based on the observation that the terms \( \nabla^{\rm sym}_{\mathcal{M}(t)} V^{||} \) and \( u \sff \) are involved in \eqref{eq:embed-saddle-PDE-b}.    
\end{remark}

\begin{remark} (Well-posedness of \eqref{eq:embed-saddle-PDE})
    \upshape
    The well-posedness of \eqref{eq:embed-saddle-PDE} for general surfaces remains a significant challenge to analyze; our current analysis is limited to planar 2D parameter domains. The analysis on planar domains demonstrates that the positive definiteness of the second fundamental form $\sff$ is a critical property. 
    Nevertheless, the embedding PDE \eqref{eq:embed-saddle-PDE} remains applicable for computations and provides an efficient method to simulate and visualize the solution of the Ricci flow. 

\end{remark}

We approximate the evolving surface \(\Man(t)\) by a piecewise flat triangular surface \(M_h(t) = \bigcup_{K \in \mathcal{T}_h(t)} K\), where \(\mathcal{T}_h(t)\) denotes the set of triangles and \(\mathcal{E}_h(t)\) denotes the set of edges on the surface \(M_h(t)\).
We employ the \( P^1 \) vector-valued Lagrange element space \( V_h^{1}(M_h(t))^3 \) on $M_h(t)$, defined by 
\[ 
V_h^{1}(M_h(t)) = \{ w \in C(M_h(t)) \mid w|_K \in P_1(K), \forall K \in \mathcal{T}_h(t)  \}.
\]
The lowest-order Regge element space on $M_h(t)$ is given by 
\[ 
\Sigma_h^{0}(M_h(t)) = \{ \sigma \in \prod_{K \in \mathcal{T}_h(t)} P_0S_2^0(K) \mid 
i_{K_1,e}^*(\sigma|_{K_1}) = i_{K_2,e}^*(\sigma|_{K_2}), \forall e = K_1 \cap K_2 \in \mathcal{E}_h(t)\}.
\]
We use $\varphi_h(t) = \varphi_h(\cdot, t) \in V_h^1(M_h(0))^3$ to denote the discrete flow map that maps $M_h(0) = M_h$ to $M_h(t)$. 
{Since the evolving discrete surface $M_h(t)$ is piecewise flat, we denote by $n_h$ the piecewise constant normal vector on $M_h(t)$, and note that its second fundamental form vanishes on each triangle. Therefore, on each triangle $K $ of $ M_h(t)$, we have
\[
\nabla_{M_h(t)}^{\mathrm{sym}} v_h = \nabla_{M_h(t)}^{\mathrm{sym}} v_h^{\parallel}, \qquad \forall v_h \in V_h^{1}(M_h(t))^3,
\]
where $v_h^{\parallel} = v_h - (n_h \cdot v_h) n_h$ denotes the tangential component of $v_h$.}
We therefore define the bilinear form $a_{M_h(t)}(\cdot,\cdot)$ on $V_h^1(M_h(t))^3$ by
\[
a_{M_h(t)}(w_h, u_h) = 
\left(\nabla^{\mathrm{sym}}_{M_h(t)} w_h, \nabla^{\mathrm{sym}}_{M_h(t)} u_h\right)_{M_h(t)} + 
\left(w_h \cdot n_h, u_h \cdot n_h\right)_{M_h(t)}, 
\]
where $w_h, u_h \in V_h^1(M_h(t))^3.$

Now, given an evolving metric $g_h(t) \in \Sigma_h$ on $M_h(0) = M_h$, the parametric FEM for \eqref{eq:embed-saddle-PDE} is to find $(\varphi_h(\cdot, t), v_h(\cdot, t), \sigma_h(\cdot, t)) \in V_h^1(M_h(0))^3 \times V_h^1(M_h(t))^3 \times \Sigma_h^{0}(M_h(t)) $ such that 
\begin{subequations}\label{eq:numerical-embedding}  
\begin{align}
\label{eq:numerical-embedding-a}
& \partial_t \varphi_h = v_h \circ \varphi_h, \\
\label{eq:numerical-embedding-b} 
&(\nabla^{\rm sym}_{M_h(t)} v_h, \nabla^{\rm sym}_{M_h(t)} w_h)_{M_h(t)} + 
    (v_h\cdot n_h,w_h \cdot n_h)_{M_h(t)} + 
    (\sigma_h, \nabla^{\rm sym}_{M_h(t)} w_h)_{M_h(t)} = 0, 
    \\ 
\label{eq:numerical-embedding-c} 
&(\nabla^{\rm sym}_{M_h(t)} v_h, \rho_h )_{M_h(t)} = \frac{1}{2}
((\varphi_{h}(t))_* \dot{g}_h, \rho_h)_{M_h(t)}, 
\end{align}
\end{subequations}
holds for all $ w_h \in V_h^{1}(M_h(t))^3$ and $ \rho_h \in \Sigma_h^{0}(M_h(t))$, 
where $(\varphi_{h}(t))_* = ((\varphi_{h}(t))^{-1})^*$ is the pushforward under the discrete flow map $\varphi_h(t): M_h(0) \rightarrow M_h(t)$.


\section{Numerical experiments}\label{sec:numerical}

In this section, we present numerical examples to illustrate the convergence of the proposed schemes \eqref{eq:num-Ric} and \eqref{eq:num-Ric-L2}, as well as the simulation of the Ricci flow. 
{In all the following numerical experiments, the time discretization is performed using a three-step linearly implicit BDF scheme, with a sufficiently small time step to ensure that the time-discretization error is negligible compared with the spatial-discretization error.} 
The numerical experiments are implemented using the open-source finite element library NGSolve \cite{schoberl2014c++}.

\begin{example}[Convergence rates for metric and Gauss curvature]
    \upshape
In the first example, we test the convergence rates in computing the metric and Gauss curvature in the Ricci flow of a revolution metric, which was considered in \cite{rubinstein2005visualizing}. 
For each point $(x,y,z) $ on the the unit sphere $\mathbb{S}^2\subset\mathbb{R}^3$, we define its coordinates $(s,\theta)$ by 
    \[
    s = \frac{\pi}{2} -  \arcsin(z) \in [0, \pi], \quad \text{and} \quad \theta =  \arctan(y/x) \in [0, 2\pi).
    \]
In these coordinates, the revolution metric on $\mathbb{S}^2$ evolving under the Ricci flow can be written as 
\[
g(t;s,\theta)= h(t;s)\,\mathrm{d}s^2 + m(t;s)\,\mathrm{d}\theta^2 .
\]
The coefficient $h$ satisfies
\begin{subequations}\label{eq:h-PDE}
\begin{align}
\partial_t h
= \partial_s\!\left(\frac{\partial_s h}{h}\right)
+ \frac{\partial_s q}{2q}\,\frac{\partial_s h}{h}
+ \left(\frac{\partial_s^2 q}{q} - \frac{1}{2}\frac{(\partial_s q)^2}{q^2}\right),
\label{eq:h-equ}
\end{align}
\end{subequations}
where $q(s)=m(s,0)/h(s,0)$, together with the boundary conditions
\(\partial_s h(0,t)=\partial_s h(\pi,t)=0.\)
Then $m$ is determined by $m(s,t) = (m(s,0)/h(s,0)) h(s,t).$
We consider the following initial values: 
$h(s,0)=\sqrt{\cos^2(s)+9\sin^2(s)}$ and $m(s,0)=\sin^2(s)$ for $s\in[0,\pi]$.

The reference solution \( \hat{h}(s,T) \) and $\hat{m}(s,T) = \frac{m(s,0)}{h(s,0)} \hat{h}(s,T) $ is obtained by solving equation \eqref{eq:h-PDE} using the Lagrange finite elements of degree \( 5 \) and the \( 3 \)-step linearly implicit BDF, with sufficiently small mesh size of \( 2.5 \times 10^{-3} \) and time stepsize of \( 10^{-5} \). The reference solutions for the metric and Gauss curvature are defined as 
    \[
    \hat{g} = \hat{h}(s)\mathrm{d}s^2 + \hat{m}(s)\mathrm{d}\theta^2, \quad
    \hat{\kappa} = -\frac{(\partial_s \hat{m})^2}{2\hat{m}^2\hat{h}} + \frac{\partial_s^2 \hat{m}}{\hat{m}\hat{h}} - \frac{\partial_s \hat{m} \partial_s \hat{h}}{2\hat{m}\hat{h}^2}.
    \]
The errors $\hat e_g=g_h-\hat{g}^{(-l)}$ and $\hat e_\kappa={\kappa_h}-\hat{\kappa}^{(-l)}$ of the numerical solutions at \( T = 0.05 \) are presented in Figure \ref{fig:revolution} for several different mesh sizes \( h = 0.25, 0.2, 0.15, 0.1, 0.08, 0.06 \), with a sufficiently small time stepsize of $10^{-3}$ which guarantees that the error from time discretization can be neglected in observing the convergence rates with respect to $h$. 
The numerical results in Figure \ref{fig:revolution} indicate that the errors of the numerical solutions are approximately \( O(h^{q+1} + h^{r+1}) \), which is consistent with the theoretical results proved in Theorems~\ref{thm:main} and~\ref{thm:main-L2} (up to a logarithmic factor that is typically difficult to observe in numerical experiments). In the case \((r, q) = (0, 1)\), we observe that \(\hat{e}_\kappa = O(h^2)\), which is one order higher than the result established in Theorem~\ref{thm:main}. { A rigorous proof of this special case would be both interesting and challenging, and may require techniques developed in \cite{sun2021analysis,gao2023optimal},} which address optimal-order error analysis for coupled systems of nonlinear PDEs involving incompatible finite element degrees.  {The superconvergence results in~\cite[Theorem 6.5]{gopalakrishnan2023analysis} may also be relevant.}

\begin{figure}[htbp]
    \centering
    \begin{subfigure}{0.47\textwidth}
        \includegraphics[width=\textwidth]{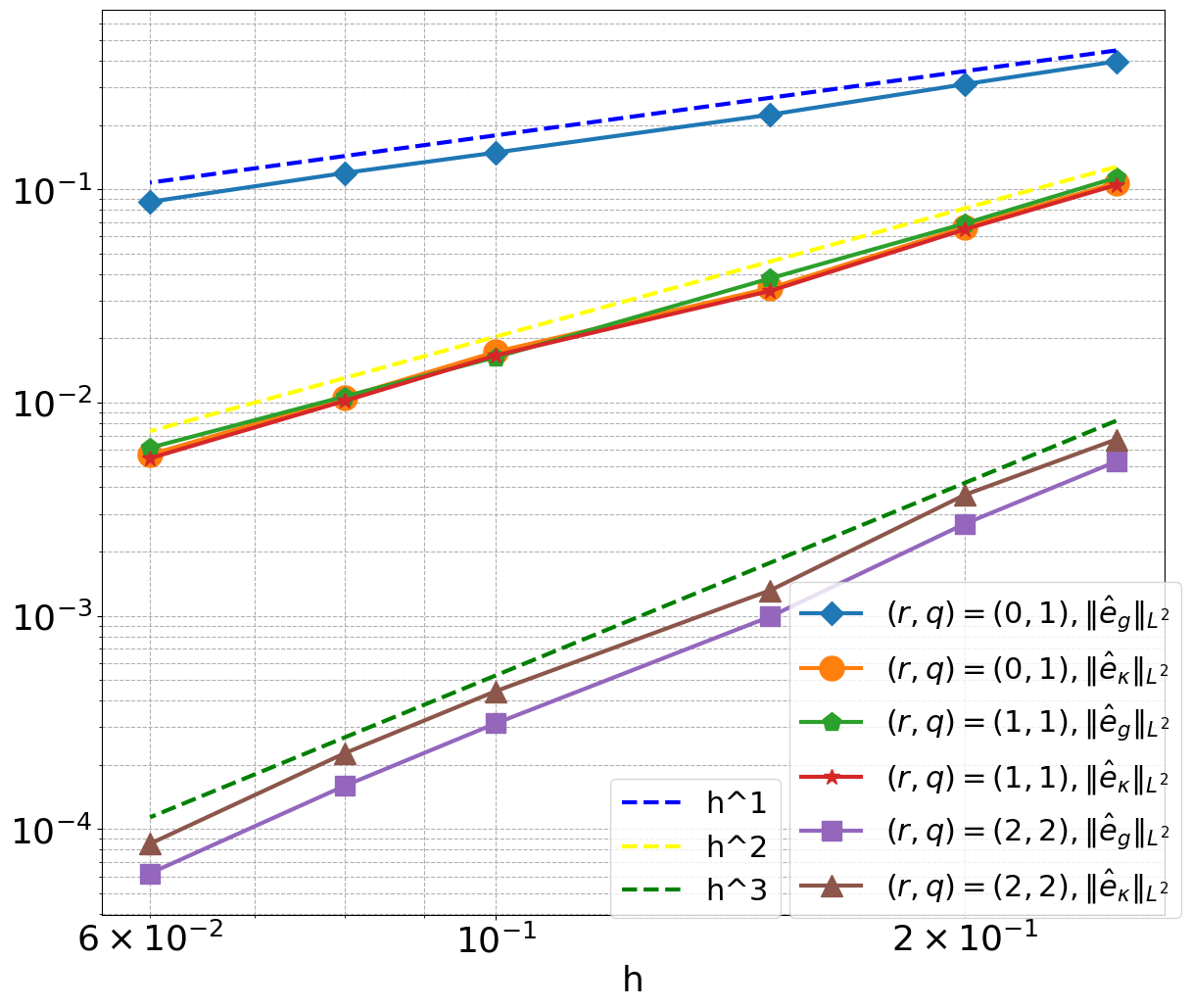}
        \caption{Errors by the numerical scheme \eqref{eq:num-Ric}}
    \end{subfigure}
    \begin{subfigure}{0.47\textwidth}
        \includegraphics[width=\textwidth]{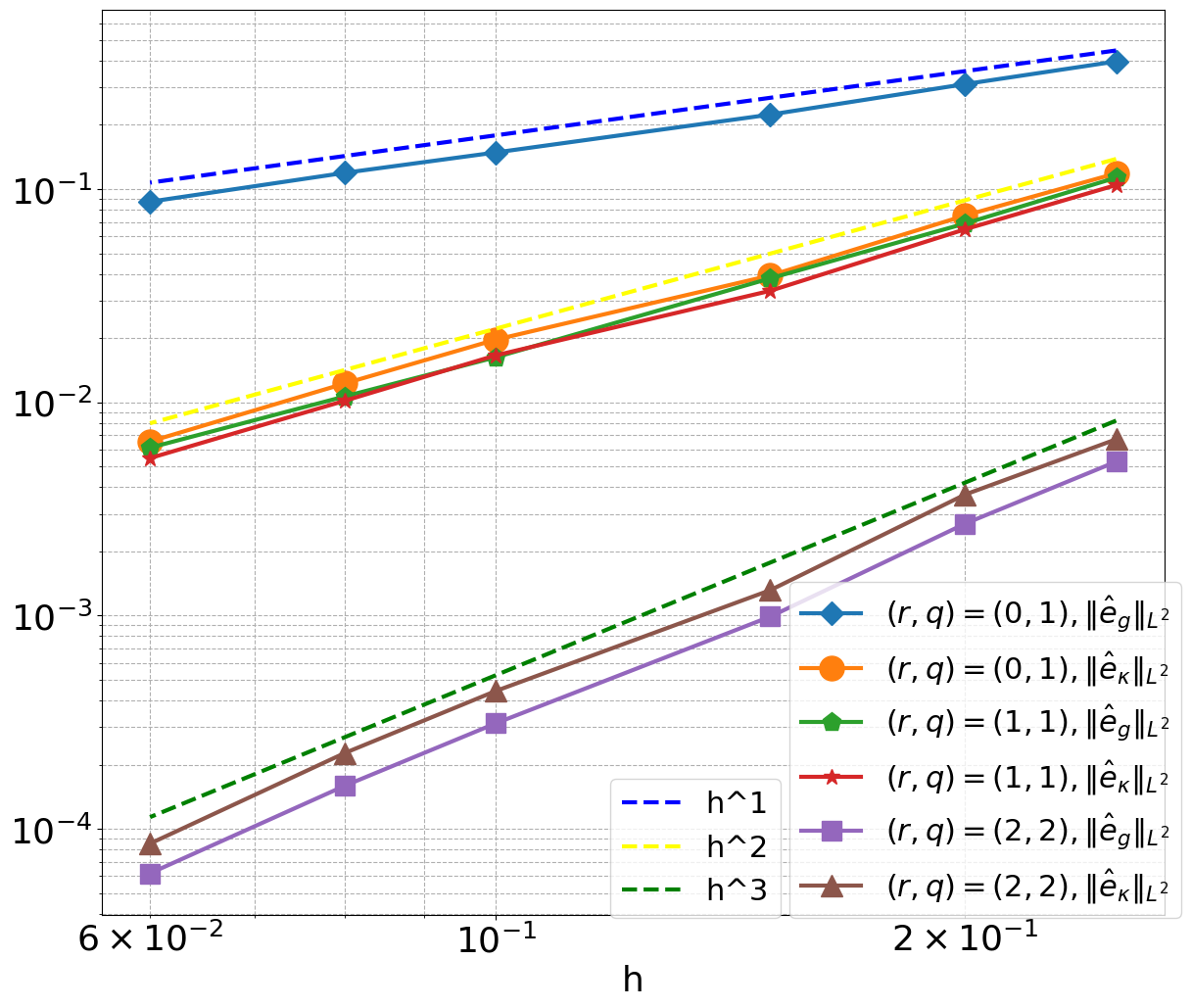}
        \caption{Errors by the numerical scheme \eqref{eq:num-Ric-L2}}
    \end{subfigure}
    \caption{Convergence rates for revolution metric with time stepsize $10^{-3}$. }
    \label{fig:revolution}
\end{figure}

\end{example}

\begin{example}[Numerical simulation of normalized Ricci flow] 
    \label{ex:ellipsoid}
    \upshape 
In the second example, we present numerical simulation and its visualization of the normalized Ricci flow with the following three different initial surfaces: 
{\small 
\begin{subequations}
    \begin{align}
    \label{eq:ellipsoid} 
    \Man &= \Bigg\{ \left( \begin{array}{c}
        \frac{1}{2}\sin \varphi \cos \theta \\[5pt]
        \frac{1}{2} \sin \varphi \sin \theta \\[5pt]
        \cos \varphi
    \end{array} \right) \colon
    \quad \theta \in [0, 2\pi), \quad \varphi \in [0, \pi] \Bigg\}, \\ 
    \label{eq:surfaceA} 
    \Man &= \Bigg\{ \left( \begin{array}{c}
        ( 0.7 \sin \varphi + 0.1 \sin(2\varphi) ) \cos \theta \\[5pt]
        ( 0.7 \sin \varphi + 0.1 \sin(2\varphi) ) \sin \theta \\[5pt]
        0.5 \cos \varphi
    \end{array} \right) \colon
    \quad \theta \in [0, 2\pi), \quad \varphi \in [0, \pi] \Bigg\}, \\ 
    \label{eq:surfaceB}
    \Man &= \Bigg\{ \left( \begin{array}{c}
        2.5\cos(\varphi) \cos \theta \\[5pt]
        2.5\cos(\varphi) \sin \theta \\[5pt]
        (1+0.6(\cos^2(\varphi)-1)^2) \sin(\varphi)
    \end{array} \right) \colon
    \quad \theta \in [0, 2\pi), \quad \varphi \in [-\pi/2, \pi/2] \Bigg\}. 
    \end{align}
\end{subequations}}
The initial values of the metric and Gauss curvature are set using the induced Euclidean metric of the initial surfaces, and the time stepsize is chosen as \( 10^{-3} \). To visualize the Ricci flow, we numerically embed the solution into \( \mathbb{R}^3 \) using the algorithm described in Section~\ref{sec:embedding} (see \eqref{eq:numerical-embedding}). The simulation results of the embedded Ricci flow are shown in Figure~\ref{fig:simulation}. These results demonstrate the effectiveness of the proposed method for simulating the Ricci flow embedded in \( \mathbb{R}^3 \). A rigorous analysis of the embedding algorithm is beyond the scope of this paper and remains an interesting and challenging topic.

\begin{figure}[htbp]
        \centering
        \begin{subfigure}{0.3\textwidth}
            \centering
            \includegraphics[width=0.8\textwidth]{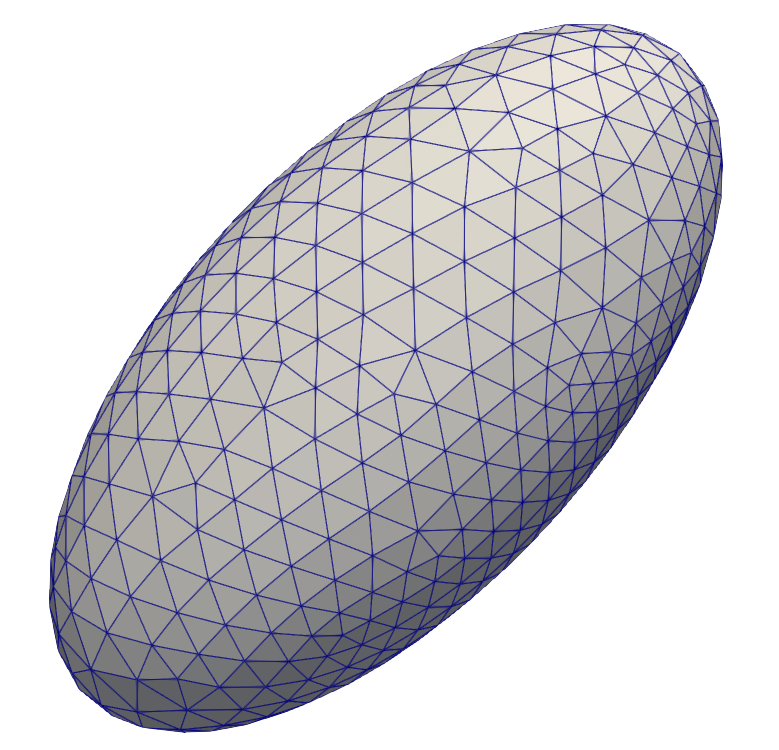}
            \caption{Initial surface in \eqref{eq:ellipsoid}}
        \end{subfigure}
        \hspace{10pt}
        \begin{subfigure}{0.24\textwidth}
            \centering
            \includegraphics[width=\textwidth]{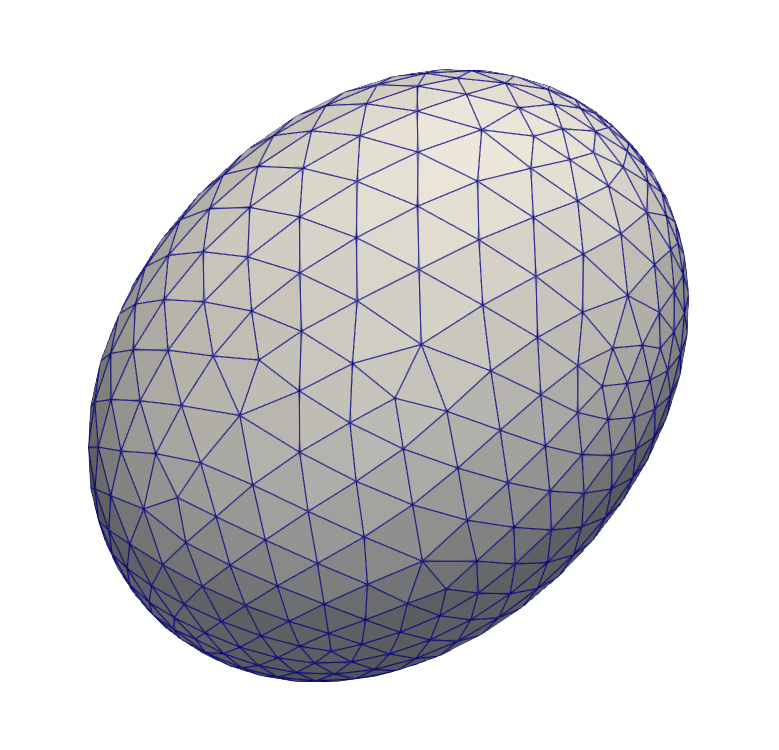}
            \caption{Surface at $t = 0.15$}
        \end{subfigure}
        \hspace{10pt}
        \begin{subfigure}{0.24\textwidth}
            \centering
            \includegraphics[width=\textwidth]{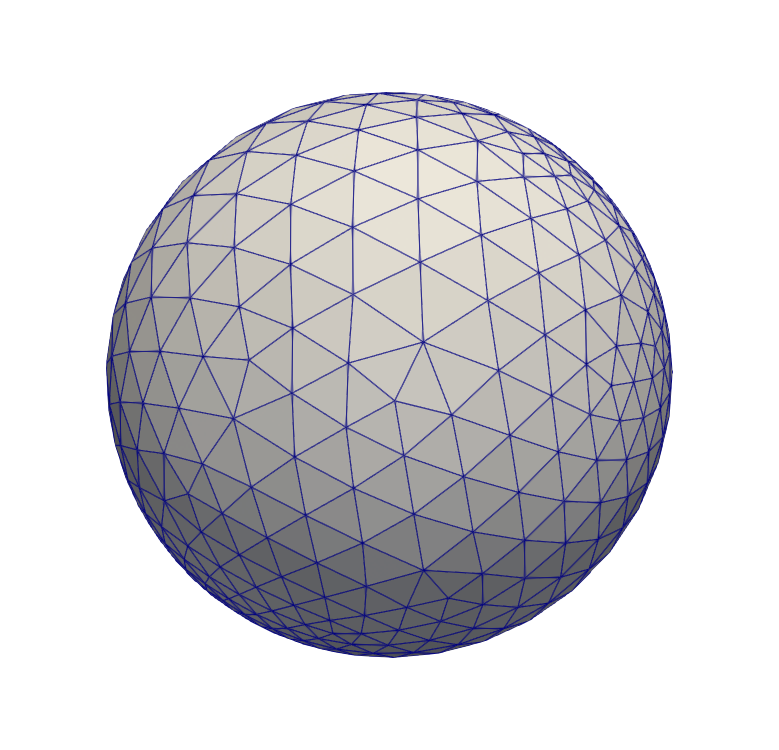}
            \caption{Surface at $t = 0.7$}
        \end{subfigure}
    
        \begin{subfigure}{0.3\textwidth}
            \centering
            \includegraphics[width=0.8\textwidth]{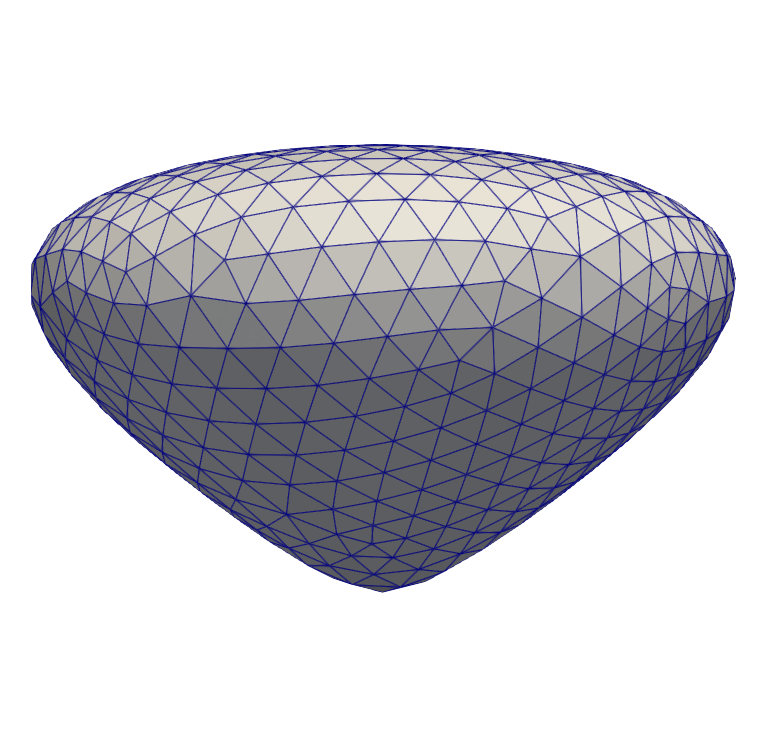}
            \caption{Initial surface in \eqref{eq:surfaceA}}
        \end{subfigure}
        \hspace{10pt}
        \begin{subfigure}{0.24\textwidth}
            \centering
            \includegraphics[width=\textwidth]{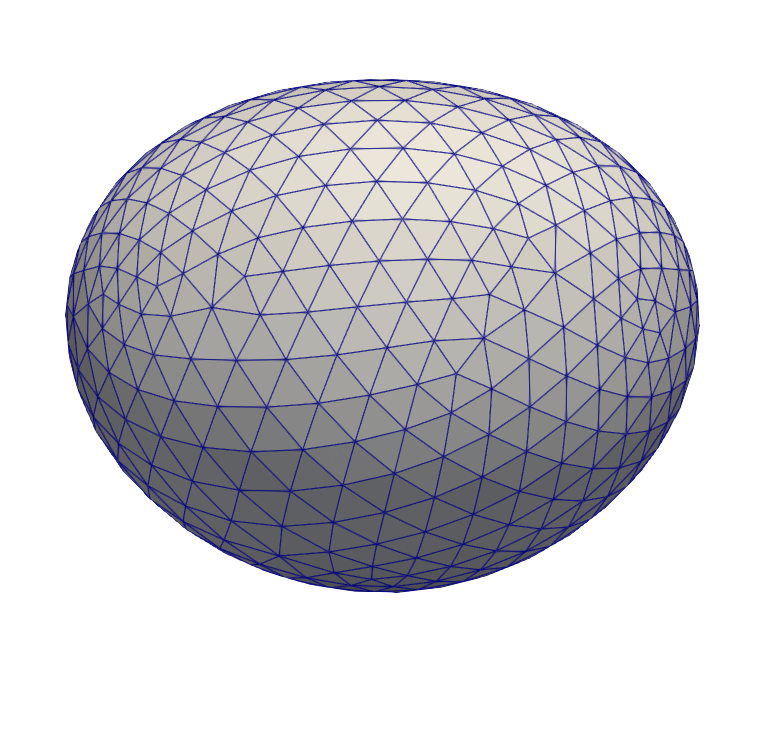}
            \caption{Surface at $t = 0.06$}
        \end{subfigure}
        \hspace{10pt}
        \begin{subfigure}{0.24\textwidth}
            \centering
            \includegraphics[width=\textwidth]{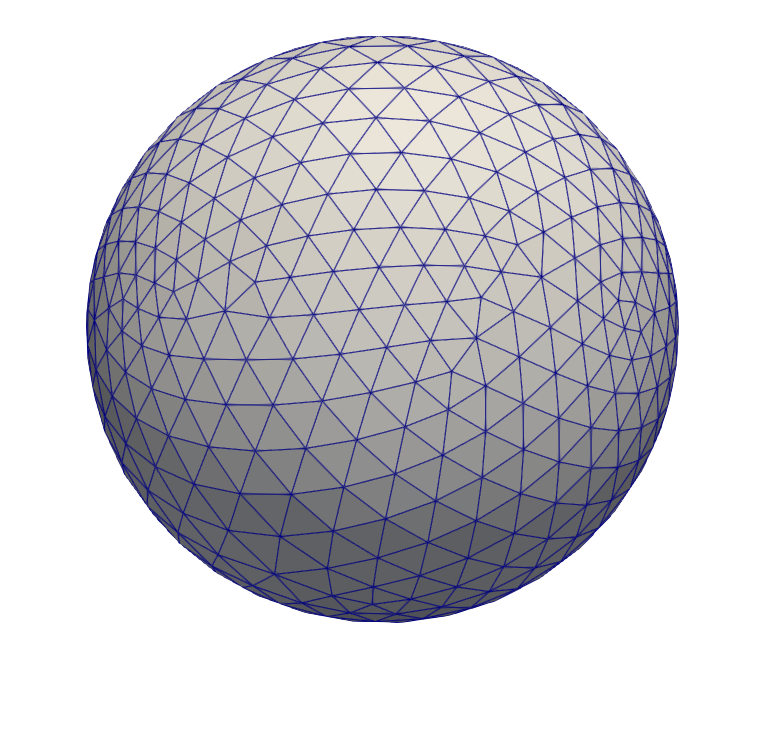}
            \caption{Surface at $t = 0.4$}
        \end{subfigure}
    
        \begin{subfigure}{0.3\textwidth}
            \centering
            \includegraphics[width=0.8\textwidth]{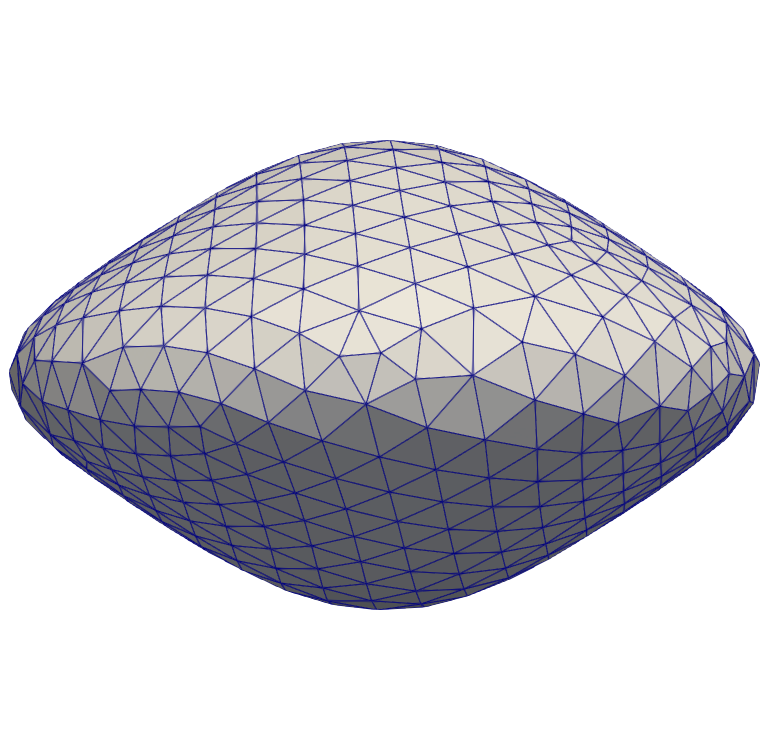}
            \caption{Initial surface in \eqref{eq:surfaceB}}
        \end{subfigure}
        \hspace{10pt}
        \begin{subfigure}{0.24\textwidth}
            \centering
            \includegraphics[width=\textwidth]{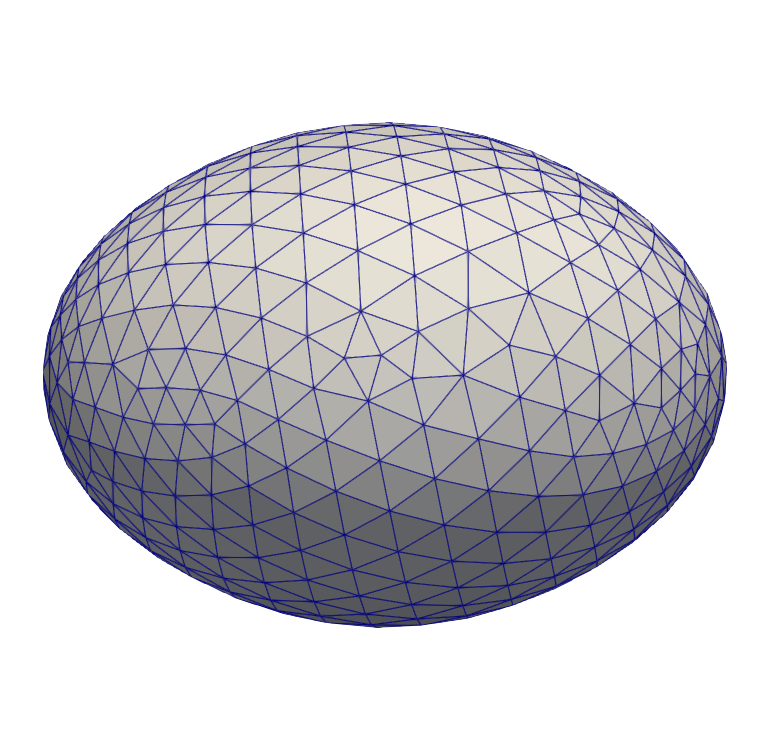}
            \caption{Surface at $t = 0.3$}
        \end{subfigure}
        \hspace{10pt}
        \begin{subfigure}{0.24\textwidth}
            \centering
            \includegraphics[width=\textwidth]{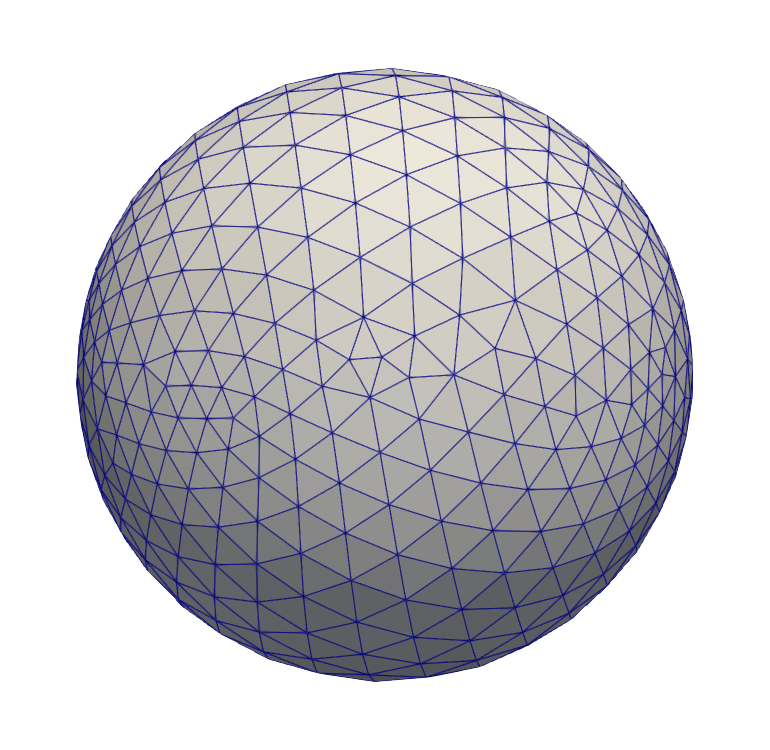}
            \caption{Surface at $t = 4.3$}
        \end{subfigure}
    
\caption{Numerical simulation of normalized Ricci flow.}
        \label{fig:simulation}
\end{figure}

\end{example}

    

\section{Conclusion} 

We have proved the convergence of a finite element discretization for the two-dimensional Ricci flow by reformulating it as a solution-driven metric evolution. In this formulation, the evolution of the metric is governed by the Gaussian curvature, which satisfies a parabolic equation depending on the evolving metric. The proposed numerical method preserves key geometric structures of the Ricci flow, including area conservation and the Gauss--Bonnet theorem. To analyze the convergence of the numerical solutions, we adapted a matrix-vector formulation originally developed for solution-driven surface evolution under extrinsic curvature flows and leveraged the parabolic structure of the equation governing the Gaussian curvature. 
{Our FEM semi-discretization yields a DAE system that we solved with linearly implicit BDF methods.  A fully discrete error analysis should be possible if one combines our spatial estimates with established techniques (see, e.g., \cite[(5.43)]{dziuk2013finite} or \cite[Section~6]{kovacs2019convergent}). We plan to further investigate this in future work.}

{
Evolution equations for intrinsic curvature are a central tool in the analysis of intrinsic curvature flows.
For many flows---such as the Calabi flow, the 3D Yamabe flow, and the 3D Ricci flow---the intrinsic curvature evolution typically has a parabolic structure.
Therefore the analytical framework presented here—including the reformulation into a solution-driven metric evolution and the error analysis based on the matrix–vector formulation—offers a promising approach for studying the convergence of numerical approximations of other intrinsic geometric flows. This will be the focus of our future work.
} 


\bibliographystyle{abbrv}
\bibliography{reference}

\newpage
\appendix

\begin{center}
{\large\bf Supplementary material} \vspace{10pt}
\end{center}
\renewcommand{\thesection}{\Alph{section}}

This supplementary material complements the paper by providing some fundamental results that will be used in this paper. 

\section{{Error estimates for the time derivative of the Ritz projection}} \label{sec:Ritz}
\renewcommand{\theequation}{A.\arabic{equation}}
\renewcommand{\thesubsection}{A.\arabic{subsection}}

{This section proves the error estimates for the time derivative of the Ritz projection, \(\partial_t R_h u\), stated in Lemma \ref{lm:app-Ritz}.
The \(H^1\) and \(L^2\) error estimates in \eqref{eq:app-Ritz-H1} and \eqref{eq:app-Ritz-L2} are derived in the following two parts.
}

\noindent{\it Part 1.} \textbf{\( H^1 \) error estimate of  \( \partial_t R_h u(t) \)}: 
We  differentiate \eqref{eq:Ritz} with respect to time: 
\begin{subequations}\label{eq:Ritz-1-pt}  
\begin{align}
\label{eq:Ritz-1-pt-a} 
\hspace{-3mm}
(\nabla_{g_h^*} \partial_t R_h u ,& \nabla_{g_h^*} \phi_h)_{g_h^*, M_h} + 
\int_{M_h} \partial_t( \sqrt{\det g_h^*}(g_h^*)^{-1}) \nabla R_h u \cdot \nabla \phi_h\mathrm{d}x \notag \\ 
& = (\nabla_{g} \partial_t u , \nabla_{g} \phi_h^{(l)})_{g, \mathcal{M}} +  
\int_{M_h} \partial_t( \sqrt{\det g^{(-l)}}(g^{(-l)})^{-1}) \nabla u^{(-l)} \cdot \nabla \phi_h\mathrm{d}x, 
\end{align}
\begin{align} 
\label{eq:Ritz-1-pt-b} 
(( \partial_t R_h u )^{(l)}, 1)_{\mathcal{M}} &= (\partial_t u , 1)_{\mathcal{M}}.
\end{align}
\end{subequations} 
Recall $I_h u^{(-l)}(t) \in V_h$ is the Lagrange interpolant.  It satisfies 
\[ \partial_t I_h u^{(-l)} = I_h \partial_t u^{(-l)} = I_h (\partial_t u)^{(-l)} \in V_h. \]
Substituting $\partial_t I_h u^{(-l)}$ into \eqref{eq:Ritz-1-pt-a}, we have 
\[ 
\begin{aligned}
    & \quad (\nabla_{g_h^*} (\partial_t R_h u  - \partial_t I_h u^{(-l)}) ,\nabla_{g_h^*} \phi_h)_{g_h^*, M_h} \\ 
    & = (\nabla_{g} \partial_t u , \nabla_{g} \phi_h^{(l)})_{g, \mathcal{M}}
    - (\nabla_{g_h^*} (\partial_t u)^{(-l)} , \nabla_{g_h^*} (\phi_h))_{g_h^*, {M_h}} \\ 
    & + (\nabla_{g_h^*} ((\partial_t u)^{(-l)}- \partial_t I_h u^{(-l)}) , \nabla_{g_h^*} (\phi_h))_{g_h^*, {M_h}}\\ 
    & + \int_{M_h} 
    \big(\partial_t( \sqrt{\det g^{(-l)}}(g^{(-l)})^{-1}) \nabla u^{(-l)} 
    - \partial_t( \sqrt{\det g_h^*}(g_h^*)^{-1}) \nabla R_h u\big)
    \cdot \nabla \phi_h \mathrm{d}x\\ 
    &= I_1 + I_2 + I_3. 
\end{aligned}
\]
We use \eqref{eq:A-g-g*} in Lemma \ref{lm:norm-equ} (Norm equivalence of $g_h^*$) to derive an estimate for \( I_1 \): 
\[ 
\begin{aligned}
    I_1 & = (\nabla_{g} \partial_t u , \nabla_{g} \phi_h^{(l)})_{g, \mathcal{M}}
    - (\nabla_{g_h^*} (\partial_t u)^{(-l)} , \nabla_{g_h^*} (\phi_h))_{g_h^*, {M_h}} \\ 
    & \lesssim h^{r+1} \|\nabla (\partial_t u)^{(-l)} \|_{L^2(M_h)} 
    \|\nabla (\phi_h) \|_{L^2(M_h)}.
\end{aligned}
\]
For \( I_2 \), we employ the interpolation approximation property of \( \partial_t I_h u^{(-l)} = I_h (\partial_t u)^{(-l)} \in V_h \):
\[ 
\begin{aligned}
    I_2 &= (\nabla_{g_h^*} ((\partial_t u)^{(-l)}- \partial_t I_h u^{(-l)}) , \nabla_{g_h^*} (\phi_h))_{g_h^*, {M_h}} 
    \lesssim h^q  \|\nabla (\phi_h) \|_{L^2(M_h)}. 
\end{aligned}
\] 
For \( I_3 \), we utilize \eqref{eq:det-inv-app} and \eqref{eq:prop-g_h^*} alongside the \( H^1 \) error estimate of the Ritz projection \eqref{eq:H1-Ritz}: 
\[ 
\begin{aligned}
    I_3 & = \int_{M_h} 
\big(\partial_t( \sqrt{\det g^{(-l)}}(g^{(-l)})^{-1}) \nabla u^{(-l)} 
- \partial_t( \sqrt{\det g_h^*}(g_h^*)^{-1}) \nabla R_h u\big)
\cdot \nabla \phi_h \mathrm{d}x\\ 
& \lesssim (h^{r+1} + h^q) \|\nabla (\phi_h) \|_{L^2(M_h)}.
\end{aligned}
\]
Let \( \phi_h = (\partial_t R_h u - \partial_t I_h u^{(-l)}) \). By combining the estimates for \( I_1 \), \( I_2 \), and \( I_3 \), we obtain the \( H^1 \) error estimate for \( \partial_t R_h u(t) \): 
\begin{equation}\label{eq:H1-Ritz-pt} 
    \| \nabla( ( \partial_t u(t))^{(-l)} - \partial_t R_h u(t) ) \|_{L^2(M_h)} \lesssim h^q + h^{r+1}.
\end{equation} 
\noindent{\it Part 2.} \textbf{\( L^2 \) error estimate of \( \partial_t R_h u(t) \)}: 
Let  
\[ \Eh(t) = \partial_t R_h u(t) - \partial_t u^{(-l)}(t), \quad \text{and} \quad ((\Eh(t))^{(l)}, 1)_{\mathcal{M}} = 0, \]  
where the last condition follows from \eqref{eq:Ritz-1-pt-b}. Consider the following elliptic problem: find \( \Psi: \mathcal{M} \to \mathbb{R} \) such that  
\[ 
(\nabla_g \Psi, \nabla_g v)_{g, \mathcal{M}} = ( \Eh^{(l)}, v )_{\mathcal{M}}, 
\quad \forall v \in H^1(\mathcal{M}). 
\]  
According to the regularity theory of elliptic equations, we have \( \Psi \in H^2(\mathcal{M}) \) and it satisfies
\begin{equation}\label{eq:ell-regular-pt} 
\|\Psi\|_{H^2(\mathcal{M})} \lesssim \| (\Eh)^{(l)} \|_{L^2(\mathcal{M})} \lesssim \|\Eh\|_{L^2(M_h)}.  
\end{equation}
Using  \eqref{eq:g-g*-equi-L2} in Lemma \ref{lm:norm-equ} (Norm equivalence of $g_h^*$), we have
\[ 
\begin{aligned}
    \|\Eh\|_{L^2(M_h)}^2 & \eqsim ( (\Eh)^{(l)} , (\Eh)^{(l)}  )_{\mathcal{M}} 
    = (\nabla_g \Psi, \nabla_g (\Eh)^{(l)} )_{g, \mathcal{M}} \\ 
    & = (\nabla_g \Psi, \nabla_g ( \partial_t R_h u)^{(l)} )_{g, \mathcal{M}}
    - (\nabla_g \Psi, \nabla_g \partial_t u )_{g, \mathcal{M}} \\ 
    & = (\nabla_g \Psi, \nabla_g ( \partial_t R_h u)^{(l)} )_{g, \mathcal{M}} - 
    (\nabla_{g_h^*} (\Psi)^{(-l)}, \nabla_{g_h^*} (\partial_t R_h u) )_{g_h^*, {M_h}} \\ 
    & + (\nabla_{g_h^*} (\Psi)^{(-l)}, \nabla_{g_h^*} (\partial_t R_h u) )_{g_h^*, {M_h}}
    - (\nabla_g \Psi, \nabla_g \partial_t u )_{g, \mathcal{M}} = I_1 + I_2. 
\end{aligned}
\]
We use \eqref{eq:A-g-g*} in Lemma \ref{lm:norm-equ} (Norm equivalence of $g_h^*$) to derive an estimate for $I_1$:
\[ \begin{aligned}
    I_1 & = (\nabla_g \Psi, \nabla_g ( \partial_t R_h u)^{(l)} )_{g, \mathcal{M}} - 
    (\nabla_{g_h^*} (\Psi)^{(-l)}, \nabla_{g_h^*} (\partial_t R_h u) )_{g_h^*, {M_h}} \\ 
    & \lesssim h^{r+1} \| \nabla  (\partial_t R_h u )\|_{L^2(M_h)} \| \nabla (\Psi)^{(-l)}\|_{L^2(M_h)}
    \lesssim h^{r+1}  \|\Eh\|_{L^2(M_h)}. 
\end{aligned} \]
For $I_2$, recall \eqref{eq:Ritz-1-pt-a}. We have
\[ 
\begin{aligned} 
        (\nabla_{g_h^*} & \partial_t R_h u , \nabla_{g_h^*} I_h (\Psi)^{(-l)})_{g_h^*, M_h} 
        - (\nabla_{g} \partial_t u , \nabla_{g} (I_h (\Psi)^{(-l)})^{(l)})_{g, \mathcal{M}} \notag \\ 
        & = \int_{M_h} \big(\partial_t( \sqrt{\det g^{(-l)}}(g^{(-l)})^{-1}) \nabla u^{(-l)}  
        - \partial_t( \sqrt{\det g_h^*}(g_h^*)^{-1}) \nabla R_h u \big)\cdot \nabla I_h (\Psi)^{(-l)}\mathrm{d}x.
 \end{aligned}
\]
Substituting it into $I_2$, we have
\[\begin{aligned}
I_2 &= (\nabla_{g_h^*} ((\Psi)^{(-l)} - I_h (\Psi)^{(-l)}), \nabla_{g_h^*} (\partial_t R_h u) )_{g_h^*, {M_h}}
- (\nabla_g (\Psi - (I_h (\Psi)^{(-l)})^{(l)}), \nabla_g \partial_t u )_{g, \mathcal{M}} \\ 
& + \int_{M_h} \partial_t( \sqrt{\det g^{(-l)}}(g^{(-l)})^{-1}) \big(\nabla u^{(-l)} - \nabla R_h u \big)\cdot \nabla (I_h (\Psi)^{(-l)} - (\Psi)^{(-l)})\mathrm{d}x. 
\\
& + \int_{M_h} \big(\partial_t( \sqrt{\det g^{(-l)}}(g^{(-l)})^{-1})   
- \partial_t( \sqrt{\det g_h^*}(g_h^*)^{-1}) \big) \nabla R_h u \cdot \nabla I_h (\Psi)^{(-l)}\mathrm{d}x\\
& + \int_{M_h} \partial_t( \sqrt{\det g^{(-l)}}(g^{(-l)})^{-1}) \big(\nabla u^{(-l)} - \nabla R_h u \big)\cdot \nabla (\Psi)^{(-l)}\mathrm{d}x
 = A_1 + A_2 + A_3 + A_4. 
\end{aligned} \]
For $A_1$, we utilize \eqref{eq:A-g-g*} in Lemma \ref{lm:norm-equ} (Norm equivalence of $g_h^*$) and the $H^1$ error estimate of $\partial_t R_h u$ \eqref{eq:H1-Ritz-pt} alongside the elliptic regularity \eqref{eq:ell-regular-pt}: 
\[ 
\begin{aligned}
    \hspace{-5pt}
    A_{1} & = 
(\nabla_{g_h^*} ((\Psi)^{(-l)} - I_h (\Psi)^{(-l)}), \nabla_{g_h^*} (\partial_t R_h u - (\partial_t u)^{(-l)} ) )_{g_h^*, {M_h}} \\ 
& + (\nabla_{g_h^*} ((\Psi)^{(-l)} - I_h (\Psi)^{(-l)}), \nabla_{g_h^*} (\partial_t u)^{(-l)}  )_{g_h^*, {M_h}} 
- (\nabla_g (\Psi - (I_h (\Psi)^{(-l)})^{(l)}), \nabla_g \partial_t u )_{g, \mathcal{M}} \\
& \lesssim (h^{q+1} + h^{r+1}) \|\Psi\|_{H^2(\mathcal{M})}\lesssim   (h^{q+1} + h^{r+1}) \|\Eh\|_{L^2(M_h)}.
\end{aligned}
\]
For $A_2$, we utilize $H^1$ error estimate of $R_h u$ \eqref{eq:H1-Ritz} and elliptic regularity \eqref{eq:ell-regular-pt}: 
\[ 
\begin{aligned}
    A_2 &= \int_{M_h} \partial_t( \sqrt{\det g^{(-l)}}(g^{(-l)})^{-1}) \big(\nabla u^{(-l)} - \nabla R_h u \big)\cdot \nabla (I_h (\Psi)^{(-l)} - (\Psi)^{(-l)})\mathrm{d}x \\ 
    & \lesssim (h^{q+1} + h^{r+2}) \|\Psi\|_{H^2(\mathcal{M})}
    \lesssim  (h^{q+1} + h^{r+2})  \|\Eh\|_{L^2(M_h)}.
\end{aligned} 
\]
For $A_3$, we utilize \eqref{eq:det-inv-app} and \eqref{eq:prop-g_h^*} alongside the elliptic regularity \eqref{eq:ell-regular-pt}: 
\[ 
\begin{aligned}
    A_3 & = \int_{M_h} \big(\partial_t( \sqrt{\det g^{(-l)}}(g^{(-l)})^{-1})   
    - \partial_t( \sqrt{\det g_h^*}(g_h^*)^{-1}) \big) \nabla R_h u \cdot \nabla I_h (\Psi)^{(-l)}\mathrm{d}x \\ 
    & \lesssim h^{r+1} \| \nabla (\Psi)^{(-l)} \|_{L^2(M_h)} \| \nabla R_h u  \|_{L^2(M_h)} 
    \lesssim  h^{r+1}  \|\Eh\|_{L^2(M_h)}.
\end{aligned}
\]
Finally, for  $A_4$, we compute
\[ \begin{aligned}
    & \quad \partial_t( \sqrt{\det g^{(-l)}}(g^{(-l)})^{-1}) \\ 
    & = 
    ( \partial_t \sqrt{\det g^{(-l)}}) (g^{(-l)})^{-1} + 
     \sqrt{\det g^{(-l)}}(\partial_t(g^{(-l)})^{-1} )\\
     & = \frac{1}{2}  \sqrt{\det g^{(-l)}} (g^{(-l)})^{-1} : ( \partial_t g^{(-l)}) (g^{(-l)})^{-1}
     - \sqrt{\det g^{(-l)}} (g^{(-l)})^{-1}   ( \partial_t g^{(-l)})  (g^{(-l)})^{-1} \\ 
     & = \Big(\frac{1}{2}  (g^{(-l)})^{-1} : ( \partial_t g^{(-l)})  {\rm I} - (g^{(-l)})^{-1}   ( \partial_t g^{(-l)})\Big) \sqrt{\det g^{(-l)}}(g^{(-l)})^{-1} \\
     & = \sigma \sqrt{\det g^{(-l)}}(g^{(-l)})^{-1}, 
\end{aligned} \]
where we denote 
\[  \sigma(t) = \frac{1}{2}  (g^{(-l)}(t))^{-1} : ( \partial_t g^{(-l)}(t)) {\rm I} - (g^{(-l)}(t))^{-1}   ( \partial_t g^{(-l)}(t)). \]
Then, we have 
\[ \begin{aligned}
    A_4 &=  \int_{M_h} \sigma \sqrt{\det g^{(-l)}} (g^{(-l)})^{-1}
  \nabla (u^{(-l)} -  R_h u )\cdot \nabla (\Psi)^{(-l)} \mathrm{d}x \\ 
    & = ( \sigma^{(l)}  \nabla_g (u - (R_h u)^{(l)} ), \nabla_g \Psi )_{g, \mathcal{M}} \\ 
    & = - (   (u - (R_h u)^{(l)} ), \nabla_g \cdot (\sigma^{(l)}  \nabla_g \Psi) )_{g, \mathcal{M}},
\end{aligned} \]
where 
\[ \sigma^{(l)} = \big(
\frac{1}{2} (g^{(-l)})^{-1} : ( \partial_t g^{(-l)}) \delta_k^l - 
(g^{(-l)})^{l,j}   ( \partial_t g^{(-l)})_{j,k}    
\big) \mathrm{d}x^k \otimes \partial_l \]
is a $(1,1)$-tensor field on $\mathcal{M}$ with $W^{1, \infty}$ regularity. Therefore, combining with elliptic regularity \eqref{eq:ell-regular-pt} and the $L^2$ estimate of $R_h u $ \eqref{eq:L2-Ritz}, we have 
\[ A_4 \lesssim (h^{q+1} + h^{r+1}) \| \Psi\|_{H^2(\mathcal{M})}  \lesssim (h^{q+1} + h^{r+1}) \|\Eh\|_{L^2(M_h)}. \]
By combining the estimates for  \( I_1 \), \( I_2 \), and \( A_1, A_2, A_3, A_4 \), we obtain the \( L^2 \) error estimate for \( \partial_t R_h u(t) \): 
\begin{equation}\label{eq:L2-Ritz-pt} 
    \|  ( \partial_t u(t))^{(-l)} - \partial_t R_h u(t)  \|_{L^2(M_h)} \lesssim h^{q+1} +  h^{r+1}.
\end{equation}

\section{Proofs of Lemma \ref{lm:stab-Pgh-Pgh*} and Lemma \ref{lm:app-Projection}} \label{sec:P-Stab}
\renewcommand{\theequation}{C.\arabic{equation}}
\renewcommand{\thesubsection}{C.\arabic{subsection}}

This section establishes the proofs of Lemma \ref{lm:stab-Pgh-Pgh*} and Lemma \ref{lm:app-Projection}. To begin with, we first present the following norm equivalence property for the Regge element.
\begin{lemma}\label{lm:norm-equi-Regge} 
    Under the boundedness results in \eqref{eq:bound-g*} and \eqref{eq:bound-gh}, for any triangle $K \in \mathcal{T}_h$ and $s\in [1, +\infty]$, the following norm equivalence holds{\rm:} 
\begin{equation}\label{eq:norm-equ-gh} 
\|\sigma_h\|_{L^s(K)} \eqsim 
h^{2/s - 2}\sum_{\alpha = 1}^{{\rm dim}P_{r}S^0_2 (K)} 
|\N_{\alpha, g_h}(\sigma_h)|, 
\quad \forall \sigma_h \in P_{r}S^0_2 (K), t \in [0, t^*], 
\end{equation} 
where the linear functionals \(\mathcal{N}_{\alpha, g_h}: P_{r}S^0_2 (K) \rightarrow \mathbb{R}\), for \(\alpha = 1, \dots, \dim P_{r}S^0_2 (K)\), are the degrees of freedom, defined by
\begin{equation}\label{eq:N_alpha} 
    \mathcal{N}_{\alpha, g_h}(\sigma_h) = 
    (\sigma_h, \rho_{\alpha})_{g_h,K},
    \quad \text{or} \quad
    \mathcal{N}_{\alpha, g_h}(\sigma_h) = 
    \langle \sigma_h(\tau_{g_h}, \tau_{g_h}), q_\alpha \rangle_{g_h, e}.
\end{equation}
Here $\rho_{\alpha} \in P_{r-1}S^0_2 (K)$ and $q_{\alpha} \in P_{r}(e)$ satisfy 
\begin{equation}\label{eq:esti_rho_q} 
\| \rho_\alpha \|_{L^\infty(K)} \eqsim 1, 
\quad \text{and} \quad 
\| q_\alpha\|_{L^\infty(K)} \eqsim h. 
\end{equation} 
Similar results also hold for $g_h^*$, i.e., 
\begin{equation}\label{eq:norm-equ-gh*} 
    \|\sigma_h\|_{L^s(K)} \eqsim 
    h^{2/s - 2}\sum_{\alpha = 1}^{{\rm dim}P_{r}S^0_2 (K)} 
    |\N_{\alpha, g_h^*}(\sigma_h)|, 
    \quad \forall\sigma_h \in P_{r}S^0_2 (K), t \in [0, T]. 
\end{equation} 
\end{lemma}
\begin{proof}
    We provide the proof of \eqref{eq:norm-equ-gh}, and the proof of \eqref{eq:norm-equ-gh*} is similar, and therefore omitted.
    Let $F_K: \hat K \rightarrow K$ be an affine map from a reference triangle $\hat K \subset \mathbb{R}^2$ to $K \subset M_h$. 
    Denote $(F_K)_*: T\hat K \rightarrow TK$ the tangent map of $F_K$. 
    Let $\hat \partial_i, i = 1, 2,$ be a basis for the tangent space of $\hat K$.  Then $\bar  \partial_i = (F_K)_* \hat \partial_i$, $i = 1,2$, form a basis for the tangent space of $K$. Now, for $\sigma_h \in P_{r}S^0_2 (K)$, we denote its pullback as $\hat \sigma_h \in P_r(\hat K; \mathbb{R}^{2 \times 2}_{\rm sym})$, defined as: 
    \[ (\hat \sigma_h)_{i j} = \sigma_h(\bar \partial_i,\bar \partial_j)
    = \sigma_h((F_K)_* \hat \partial_i, (F_K)_* \hat \partial_j )
    = ((F_K)^* \sigma_h) (\hat \partial_i, \hat \partial_j), 
    \quad \text{on } \hat K. 
    \]
    Then, by a scaling argument \cite[Lemma 2.11]{li2018regge}, we have: 
    \begin{equation}\label{eq:K-hatK} 
    \|\sigma_h\|_{L^s(K)} \eqsim 
    \|\hat \sigma_h\|_{L^s(\hat K)} h^{2/s - 2}, \quad \forall s \in [1, \infty]. 
    \end{equation}  

    For the metric $g_h$ on $K$, we denote its pullback as $\hat g_h \in P_r(\hat K; \mathbb{R}^{2 \times 2}_{\rm sym})$, defined as: 
    \[ (\hat g_h)_{i j} = g_h(\bar\partial_i,\bar \partial_j)
    = g_h((F_K)_* \hat \partial_i, (F_K)_* \hat \partial_j )
    = ((F_K)^* g_h) (\hat \partial_i, \hat \partial_j), 
    \quad \text{on } \hat K. 
    \]
    Let $\hat \tau$ be the unit tangent vector of an edge $\hat e$ of the reference triangle $\hat K$.  Then it holds that 
    $ \tau_{g_h} = (F_K)_* \hat \tau / \sqrt{\hat \tau^{\top}\hat g_h \hat \tau}.$
    Let $(\hat g_h)^{k l}$ and $\det(\hat g_h)$ represent the components of the inverse of $\hat g_h$ and the determinant of $\hat g_h$. 
    Based on the boundedness result in \eqref{eq:bound-gh} and a scaling argument, we have 
    \begin{equation}\label{eq:hatg} 
    \begin{aligned}
        \| \hat g_h \|_{L^{\infty}(\hat K)} \eqsim 
    h^2 \|g_h\|_{L^{\infty}(K)} \eqsim h^2, \quad & 
    \|(\hat g_h)^{k l}\|_{L^{\infty}(\hat K)}  \eqsim h^{-2}, \\ 
    \| \sqrt{\hat \tau^{\top}\hat g_h \hat \tau}\|_{L^{\infty}(\hat K)}  \eqsim h, \quad  & 
    \|\det(\hat g_h)\|_{L^{\infty}(\hat K)}  \eqsim h^4.
    \end{aligned}
    \end{equation} 
    Now, for $\hat \sigma_h \in P_r(\hat K; \mathbb{R}^{2 \times 2}_{\rm sym}) $, we define the functional $\hat \N_\alpha: P_r(\hat K; \mathbb{R}^{2 \times 2}_{\rm sym}) \rightarrow \mathbb{R} $, $\alpha = 1, \dots, {\rm dim}P_r(\hat K; \mathbb{R}^{2 \times 2}_{\rm sym}) $, as
    \[ 
    \hat \N_\alpha (\hat \sigma_h) = 
    h\int_{\hat e} 
    \frac{(\hat \tau^{\top} \hat \sigma_h \hat \tau )\hat q_\alpha }{\sqrt{\hat \tau^{\top}\hat g_h \hat \tau}}
    \mathrm{d}\hat s, 
    \quad\text{or}\quad
    \hat \N_\alpha (\hat \sigma_h) = h^2\int_{\hat K} 
    \Tr(\hat g_h^{-1}\hat \sigma_h\hat g_h^{-1}\hat \rho_\alpha)
    \sqrt{\det(\hat g_h)} \mathrm{d} \hat x, 
    \] 
    where $\hat q_\alpha \in P_r(\hat e) $ and $\hat \rho_{\alpha} \in P_{r-1}(\hat K; \mathbb{R}^{2\times 2}_{\rm sym})$ are fixed functions on $\hat K$, satisfying \\
    $\|\hat q_\alpha\|_{L^{\infty}(\hat K)}  \eqsim 1$ and 
    $\|\hat \rho_\alpha\|_{L^{\infty}(\hat K)}  \eqsim 1$. 
    Based on the norm-equivalence in the finite dimensional space $P_r(\hat K; \mathbb{R}^{2 \times 2}_{\rm sym})$ and the estimate in \eqref{eq:hatg}, for any $s \in [1, \infty]$, it holds that
    \begin{equation}\label{eq:norm-equi-hatK} 
    \|\hat \sigma_h\|_{L^s(\hat K)} \eqsim 
    \sum_{\alpha = 1}^{{\rm dim}P_r(\hat K; \mathbb{R}^{2 \times 2}_{\rm sym})} 
    |\hat \N_\alpha (\hat \sigma_h)|, \quad 
    \forall \hat \sigma_{{h}} \in P_r(\hat K; \mathbb{R}^{2 \times 2}_{\rm sym}), 
    \end{equation} 
    where the hidden constant is independent of $h$.

    Define $\rho_\alpha \in P_{r-1}S^0_2 (K)$ and $q_\alpha \in P_{r}(e)$ as 
    \[ \rho_{\alpha} = h^2(\hat \rho_\alpha)_{i j}\circ(F_K)^{-1}
    \mathrm{d} x^i \mathrm{d} x^j, 
    \quad \text{and} \quad 
    q_\alpha = h \hat q_\alpha\circ (F_K)^{-1}, 
    \quad \text{on } K, \]
    where $\mathrm{d} x^i$, $i = 1,2$, are covectors satisfying $\mathrm{d} x^i(\bar \partial_j) = \delta_j^i$. Then, it holds that 
    \[ \hat \N_\alpha (\hat \sigma_h) = 
    h\int_{\hat e}
    \frac{(\hat \tau^{\top} \hat \sigma_h \hat \tau )\hat q_\alpha }{\sqrt{\hat \tau^{\top}\hat g_h \hat \tau}}\mathrm{d}\hat s= 
    \langle\sigma_h(\tau_{g_h},\tau_{g_h}),q_\alpha\rangle_{g_h,e} =  \N_{\alpha,g_h} (\sigma_h), 
    \]
    and 
    \[ \hat \N_\alpha (\hat \sigma_h) = h^2\int_{\hat K} 
    \Tr(\hat g_h^{-1}\hat \sigma_h\hat g_h^{-1}\hat \rho_\alpha)
    \sqrt{\det(\hat g_h)} \mathrm{d} \hat x = 
    (\sigma_h, \rho_\alpha)_{g_h, K} = \N_{\alpha,g_h} (\sigma_h). 
    \]
    Combining the above two formulas with \eqref{eq:norm-equi-hatK} and \eqref{eq:K-hatK}, we prove \eqref{eq:norm-equ-gh}.
    The estimate in \eqref{eq:esti_rho_q} can be proved using a scaling argument, i.e., 
    \[ \| \rho_\alpha \|_{L^\infty(K)} \eqsim 
    h^{-2}\| h^2 \hat \rho_\alpha \|_{L^\infty(\hat K)} \eqsim 1, 
    \quad \text{and} \quad 
    \| q_\alpha\|_{L^\infty(K)} \eqsim  
    \|h \hat q_\alpha\|_{L^{\infty}(\hat K)} \eqsim h.  \] 
\end{proof}

    \noindent \textit{Proof of Lemma \ref{lm:stab-Pgh-Pgh*}:} 
        Using Lemma \ref{lm:norm-equi-Regge} and the definition of $P_{g_h}$ in \eqref{eq:Regge-interpolation}, it holds that 
        \begin{equation}\label{eq:stab-inf-1} 
        \| P_{g_h } \sigma \|_{L^{\infty}(K)} \eqsim h^{-2}
         \sum_{\alpha = 1}^{{\rm dim}P_{r}S^0_2 (K)} 
        |\N_{\alpha, g_h}(P_{g_h } \sigma)| =  h^{-2}
         \sum_{\alpha = 1}^{{\rm dim}P_{r}S^0_2 (K)} 
        |\N_{\alpha, g_h}( \sigma)|. 
        \end{equation}  
        For each $\alpha = 1,\dots,{\rm dim}P_{r}S^0_2 (K)$, using \eqref{eq:esti_rho_q}, we obtain
        \begin{equation}\label{eq:stab-inf-2} 
        |\N_{\alpha, g_h}( \sigma)| = 
        \begin{cases}
            (\sigma, \rho_\alpha)_{g_h, K} \lesssim 
            \|\sigma\|_{L^{\infty}(K)} h^2, \\
            \langle \sigma(\tau_{g_h},\tau_{g_h}), q_\alpha \rangle_{g_h, e} \lesssim 
            \|\sigma\|_{L^{\infty}(e)} h^2.
        \end{cases}
        \end{equation} 
        By combining \eqref{eq:stab-inf-1} and \eqref{eq:stab-inf-2}, we prove the first estimate in \eqref{eq:stab-inf}. The proof of the second estimate in \eqref{eq:stab-inf} is similar, therefore omitted. 
    
        Now, we assume $\sigma$ is a polynomial symmetric $(0,2)$-tensor field on $K$, for $1 \leq s \leq +\infty$, using Lemma \ref{lm:norm-equi-Regge}, we have 
        \begin{equation}\label{eq:stab-Ls-1} 
            \| P_{g_h } \sigma \|_{L^{s}(K)} \eqsim h^{2/s-2}
             \sum_{\alpha = 1}^{{\rm dim}P_{r}S^0_2 (K)} 
            |\N_{\alpha, g_h}(P_{g_h } \sigma)| =  h^{2/s-2}
             \sum_{\alpha = 1}^{{\rm dim}P_{r}S^0_2 (K)} 
            |\N_{\alpha, g_h}( \sigma)|. 
        \end{equation} 
        For each $\alpha = 1,\dots,{\rm dim}P_{r}S^0_2 (K)$, using \eqref{eq:esti_rho_q}, we obtain
        \begin{equation}\label{eq:stab-Ls-2} 
        |\N_{\alpha, g_h}( \sigma)| = 
        \begin{cases}
            (\sigma, \rho_\alpha)_{g_h, K} \lesssim 
            \|\sigma\|_{L^{s}(K)} h^{2-2/s} , \\
            \langle \sigma(\tau_{g_h},\tau_{g_h}), q_\alpha \rangle_{g_h, e} \lesssim 
            \|\sigma\|_{L^{s}(e)} h^{2-1/s}
            \lesssim \|\sigma\|_{L^{s}(K)} h^{2-2/s},
        \end{cases}
        \end{equation}
        where in the last step, we utlize the trace estimate for the polynomial $(0,2)$ tensor field $\sigma$. By combining \eqref{eq:stab-Ls-1} and \eqref{eq:stab-Ls-2}, we prove the first estimate of \eqref{eq:stab-Ls}. The proof of the second estimate in \eqref{eq:stab-Ls} is similar, therefore omitted. 
    
        We now turn to the estimate \eqref{eq:stab-Pgh-Pgh*}. From Lemma \ref{lm:norm-equi-Regge}, we have 
        \begin{equation}\label{eq:stab-Pgh-Pgh*-1} 
        \begin{aligned}
            \| (P_{g_h^* } - & P_{g_h }) \sigma \|_{L^s(K)} 
            \eqsim h^{2/s-2} 
             \sum_{\alpha = 1}^{{\rm dim}P_{r}S^0_2 (K)} 
            |\N_{\alpha, g_h^*}(P_{g_h^* }\sigma) - 
            \N_{\alpha, g_h^*}(P_{g_h }\sigma)| \\ 
            & \leq h^{2/s-2} 
             \sum_{\alpha = 1}^{{\rm dim}P_{r}S^0_2 (K)} 
            |\N_{\alpha, g_h^*}(\sigma) - 
            \N_{\alpha, g_h}(\sigma)| + 
            |\N_{\alpha, g_h}(P_{g_h }\sigma) - 
            \N_{\alpha, g_h^*}(P_{g_h }\sigma)|.
        \end{aligned}
        \end{equation} 
        For the first term, it holds that 
        \[ 
        \begin{aligned}
            & \quad |\N_{\alpha, g_h^*}(\sigma) - 
            \N_{\alpha, g_h}(\sigma)| \\ 
            & = \begin{cases}
        |(\sigma, \rho_\alpha)_{g_h^*, K} - 
        (\sigma, \rho_\alpha)_{g_h, K}|  \lesssim 
        h^{2 - 2/s} 
        \| g_h  - g_h^*  \|_{L^s(K)} \|\sigma\|_{L^{\infty}(K)},
        \\
        |\langle \sigma(\tau_{g_h^*},\tau_{g_h^*}), q_\alpha \rangle_{g_h^*, e} - 
        \langle \sigma(\tau_{g_h},\tau_{g_h}), q_\alpha \rangle_{g_h, e}|  \lesssim h^{2 - 1/s} 
        \| g_h  - g_h^*  \|_{L^s(e)} \|\sigma\|_{L^{\infty}(e)}.
        \\
        \end{cases} 
        \end{aligned}
        \]
        Combining this with the trace estimate, $\| g_h  - g_h^*  \|_{L^s(e)} \lesssim h^{-1/s}\| g_h  - g_h^*  \|_{L^s(K)} $, we obtain
        \begin{equation}\label{eq:stab-Pgh-Pgh*-2} 
            |\N_{\alpha, g_h^*}(\sigma) - 
            \N_{\alpha, g_h}(\sigma)| \lesssim 
            h^{2 - 2/s} 
            \| g_h  - g_h^*  \|_{L^s(K)} \|\sigma\|_{L^{\infty}(K)}. 
        \end{equation} 
        Similarly, we have 
        \begin{equation}\label{eq:stab-Pgh-Pgh*-3} 
            \begin{aligned}
                |\N_{\alpha, g_h^*}(P_{g_h }\sigma) - 
            \N_{\alpha, g_h}(P_{g_h }\sigma)| & \lesssim h^{2 - 2/s} 
            \| g_h  - g_h^*  \|_{L^s(K)} \|P_{g_h }\sigma\|_{L^{\infty}(K)} \\
            & \lesssim h^{2 - 2/s} 
            \| g_h  - g_h^*  \|_{L^s(K)} \|\sigma\|_{L^{\infty}(K)}.  
            \end{aligned} 
        \end{equation} 
        By combining \eqref{eq:stab-Pgh-Pgh*-1}, \eqref{eq:stab-Pgh-Pgh*-2} and \eqref{eq:stab-Pgh-Pgh*-3}, we prove \eqref{eq:stab-Pgh-Pgh*}.  

    \noindent \textit{Proof of Lemma \ref{lm:app-Projection}:} 
        Let $P_h \sigma \in \Sigma_h$ be the canonical Regge interpolation with respect to the induced Euclidean metric. Then it holds that
        \( \| P_{h} \sigma - \sigma \|_{L^\infty(M_h)} \lesssim h^{r+1}. \)
        Utilizing the identity $P_{h} \sigma = P_{g_h^*(t)} P_{h} \sigma $, we have 
        \[
            \| P_{g_h^*(t)} \sigma - \sigma \|_{L^\infty(M_h)} 
            \leq \| P_{g_h^*(t)} (\sigma - P_{h} \sigma)\|_{L^{\infty}(M_h)}  
            + \| P_{h} \sigma - \sigma \|_{L^{\infty}(M_h)} 
            \lesssim h^{r+1},
        \]
        where we use the $L^\infty$ stability of $P_{g_h^*(t)}$ in \eqref{eq:stab-inf}. 
        The second estimate in \eqref{eq:app-Projection} can be proven similarly as follows:
        \[
        \begin{aligned}
        \| \Pi_{g_h^*(t)} \sigma - \sigma \|_{L^2(M_h)} & \lesssim 
        \| \Pi_{g_h^*(t)} (\sigma - P_{h} \sigma) \|_{L^2(g_h^*, M_h)} + 
        \| P_{h} \sigma - \sigma \|_{L^2(M_h)}  
         \\ 
        & \lesssim  \| \sigma - P_{g_h^*(t)} \sigma \|_{L^2(g_h^*, M_h)} + h^{r+1} 
        \lesssim  h^{r+1}.
        \end{aligned}
        \]
        

\end{document}